\documentclass[11pt]{amsart}
\usepackage{tabularx,booktabs}
\usepackage{caption}
\usepackage{amsmath}
\usepackage{amsfonts}

\usepackage{amscd}
\usepackage{amsthm}
\usepackage{amssymb} \usepackage{latexsym}
\usepackage{eufrak}
\usepackage{euscript}
\usepackage{epsfig}
\usepackage{graphics}
\usepackage{array}
\usepackage{enumerate}
\usepackage{dsfont}
\usepackage{color}
\usepackage{wasysym}
\usepackage{hyperref}
\usepackage{pdfsync}

\newcommand{\bel}[1]{\begin{equation}\label{#1}}

\newcommand{\be}{\begin{equation}}

\newcommand{\ba}{\begin{eqnarray}}
\newcommand{\ea}{\end{eqnarray}}

\newcommand{\qe}{\end{equation}}
\newcommand{\R}{{\mathbb R}}

\newcommand{\SP}{\mathbb{S}}
\newcommand{\wt}{\widetilde}

\newcommand{\Aut}{\mathrm{Aut}}

\newcommand{\PC}{\mathcal{PC}_{>0}}
\newcommand{\NNG}{\mathcal{PC}_{\geq 0}}

\newcommand{\Hmm}[1]{\leavevmode{\marginpar{\tiny%
$\hbox to 0mm{\hspace*{-0.5mm}$\leftarrow$\hss}%
\vcenter{\vrule depth 0.1mm height 0.1mm width \the\marginparwidth}%
\hbox to
0mm{\hss$\rightarrow$\hspace*{-0.5mm}}$\\\relax\raggedright #1}}}

\newtheorem{theorem}{Theorem}[section]

\newtheorem{lemma}[theorem]{Lemma}

\newtheorem{definition}[theorem]{Definition}

\newtheorem{remark}[theorem]{Remark}

\newtheorem{problem}[theorem]{Problem}
\newtheorem{example}[theorem]{Example}
\newtheorem{observation}[theorem]{Observation}

\begin{document}

\title[Planar graphs with nonnegative curvature]{The set of vertices with positive curvature in a planar graph with nonnegative curvature}
\author{Bobo Hua}
\email{bobohua@fudan.edu.cn}
\address{Bobo Hua: School of Mathematical Sciences, LMNS, Fudan University, Shanghai 200433, China; Shanghai Center for Mathematical Sciences, Fudan University, Shanghai 200433, China}

\author{Yanhui Su}
\email{suyh@fzu.edu.cn}
\address{Yanhui Su: College of Mathematics and Computer Science, Fuzhou University, Fuzhou 350116, China}

\begin{abstract} 
%On the number of vertices with positive curvature of a planar graph with nonnegative combinatorial curvature
In this paper, we give the sharp upper bound for the number of vertices with positive curvature in a planar graph with nonnegative combinatorial curvature. Based on this, we show that the automorphism group of a planar---possibly infinite---graph with nonnegative combinatorial curvature and positive total curvature is a finite group, and give an upper bound estimate for the order of the group.
\end{abstract}
\maketitle%\tableofcontents
\tableofcontents

%\author{\large  Bobo Hua$^\ast$\ \ \ \  Yanhui Su$^{\dag}$} % ����ȫ��

\par
\maketitle

\bigskip

%BB: \begin{enumerate}%\item We denote by the distance function by $d$, not by $d^G.$
%%\item Euler characteristic of open surfaces, how to define???
%\item Check citations, number of theorems!
%%\item Let be, be.
%%\item Without loss of generality, we may always consider planar graphs with nonnegative curvature.
%%\item $\partial G$ to $\partial S.$
%\end{enumerate}

\section{Introduction}\label{sec:intro}
The combinatorial curvature for a planar graph, embedded in the sphere or the plane, was introduced by \cite{MR0279280,MR0410602,MR919829,Ishida90}: Given a planar graph, one may canonically endow the ambient space with a piecewise flat metric, i.e. replacing faces by regular polygons and gluing them together along common edges. The combinatorial curvature of a planar graph is defined via the generalized Gaussian curvature of the metric surface. Many interesting geometric and combinatorial results have been obtained since then, see e.g.
\cite{MR1600371,Woess98,MR1864922,BP01,MR1894115,MR1923955,
MR2038013,MR2096789,RBK05,MR2243299,MR2299456,
MR2410938,MR2466966,MR2470818,MR2558886,MR2818734,MR2826967,MR3624614,Gh17}.

%In this paper, we only consider planar graphs.
Let $(V,E)$ be a (possibly infinite) locally finite, undirected simple graph with the set of vertices $V$ and the set of edges $E.$ It is called planar if it is topologically embedded into the sphere or the plane. We write $G=(V,E,F)$ for the combinatorial structure, or the cell complex, induced by the embedding where $F$ is the set of faces, i.e. connected components of the complement of the embedding image of the graph $(V,E)$ in the target. We say that a planar graph $G$ is a \emph{planar tessellation} if the following hold, see e.g. \cite{MR2826967}:
\begin{enumerate}[(i)]
\item Every face is homeomorphic to a disk whose boundary consists of finitely many edges of the graph.
\item Every edge is contained in exactly two different faces.
\item For any two faces whose closures have non-empty intersection, the intersection is either a vertex or an edge.
\end{enumerate} In this paper, we only consider planar tessellations and call them planar graphs for the sake of simplicity. For a planar tessellation, we always assume that for any vertex $x$ and face $\sigma,$ $$\deg(x)\geq 3,\ \mathrm{deg}(\sigma)\geq 3$$ where $\deg(\cdot)$ denotes the degree of a vertex or a face. For any planar graph $G=(V,E,F),$ we write \begin{equation}\label{eq:max}D_G:=\sup_{\sigma\in F}\deg(\sigma).\end{equation}
For a planar graph $G$, the \emph{combinatorial curvature}, the \emph{curvature} for short, at the
vertex is defined as
\begin{equation}\label{def:comb}\Phi(x)=1-\frac{\deg(x)}{2}+\sum_{\sigma\in F:x\in \overline{\sigma}}\frac{1}{\deg(\sigma)},\quad x\in V,\end{equation} where the summation is taken over all faces $\sigma$ whose closure $\overline{\sigma}$ contains $x.$  To digest the definition, we endow the ambient space, $\SP^2$ or $\R^2,$ with a canonical piecewise flat metric structure and call it the (regular) \emph{polyhedral surface}, denoted by $S(G)$: The length of each edge is set to one, each face is set to being isometric to a Euclidean regular polygon of side length one with same facial degree, and the metric is induced by gluing faces along their common edges, see \cite{MR1835418} for the definition of gluing metrics. It is well-known that the generalized Gaussian curvature on a polyhedral surface, as a measure, concentrates on the vertices. And one easily sees that the combinatorial curvature at a vertex is in fact the mass of the generalized Gaussian curvature at that vertex up to the normalization $2\pi,$ see e.g. \cite{MR2127379,MR3318509}.

In this paper, we study planar graphs with nonnegative combinatorial curvature. We denote by $$\PC:=\{G=(V,E,F): \Phi(x)>0,\forall x\in V\}$$  the class of planar graphs with positive curvature everywhere, and by $$\NNG:=\{G=(V,E,F): \Phi(x)\geq 0,\forall x\in V\}$$ the class of planar graphs with nonnegative curvature.
%This yields that $G\in \NNG$ if and only if $K\geq 0,$ i.e. $S(G)$ is a convex surface (non-negatively curved in the sense of Alexandrov), see e.g. \cite{Burago:2001dq}.
For any finite planar graph $G\in \NNG,$ by Alexandrov's embedding theorem, see e.g. \cite{MR2127379}, its polyhedral surface $S(G)$ can be isometrically embedded into $\R^3$ as a boundary of a convex polyhedron. This yields many examples for the class $\PC,$ e.g. the $1$-skeletons of 5 Platonic solids, 13 Archimedean solids, and 92 Johnson solids. Besides these, the class $\PC$ contains many other examples \cite{RBK05,MR2836763}, since in general a face of $G,$ which is a regular polygon in $S(G),$ may split into several pieces of non-coplanar faces in the embedded image of $S(G)$ in $\R^3.$

We review some known results on the class $\PC.$ Stone \cite{MR0410602}  first obtained a Myers type theorem: A planar graph with the curvature bounded below uniformly by a positive constant is a finite graph.
Higuchi \cite{MR1864922} conjectured that it is finite even if the curvature is positive everywhere, which was proved by DeVos and Mohar 
\cite{MR2299456}, see \cite{MR2096789} for the case of cubic graphs. There are two special families of graphs in $\PC$ called prisms and anti-prisms, both consisting of infinite many examples, see e.g. \cite{MR2299456}. DeVos and Mohar \cite{MR2299456} proved that there are only finitely many other graphs in $\PC$ and proposed the following problem to find the largest graph among them.
\begin{problem}[\cite{MR2299456}] What is the number $$C_{\SP^2}:=\max_{G=(V,E,F)}{\sharp V},$$ where the maximum is taken over graphs in $\PC,$ which are not prisms or antiprisms, and $\sharp V$ denotes the cardinality of $V$?
\end{problem}
On one hand, the main technique to obtain the upper bound of $C_{\SP^2}$ is the so-called discharging method, which was used in the proof of the Four Colour Theorem, see \cite{MR0543792,MR1441258}. DeVos and Mohar \cite{MR2299456} used this method to show that $C_{\SP^2}\leq 3444,$ which was improved to $C_{\SP^2}\leq 380$ by Oh \cite{MR3624614}. By a refined argument, Ghidelli \cite{Gh17} proved that $C_{\SP^2}\leq 208.$ On the other hand, for the lower bound many authors \cite{RBK05,MR2836763,Gh17,Ol17} attempted to construct large examples in this class, and finally found some examples possessing $208$ vertices. Hence, this completely answers the problem that $C_{\SP^2}=208.$

%Besides them, the other graphs in $\PC$ are called PCC graphs. 
%Since the number of vertices is always bounded above by $3444$ for any PCC graph \cite{MR2299456}. We define
%$$C_{S^2}:=\max_{G\hookrightarrow S^2}{\sharp V},\quad\quad C_{\R P^2}:=\max_{G\hookrightarrow \R P^2}{\sharp V},$$ where the maxima are taken over PCC graphs embedded into $S^2$ and $\R P^2$ respectively. DeVos and Mohar \cite{MR2299456} proposed the following problem to find the graph
%\begin{problem} What is the largest graph, i.e. having Which graph in $\PC,$ ,What are $C_{S^2}$ and $C_{\R P^2}?$ Which PCC graphs attain the maxima in either cases?
%\end{problem}

%Recently, the first part of the problem has been solved by combining the works by \cite{MR2836763,Ol17,Gh17}: $$C_{S^2}=208,\  C_{\R P^2}=104,$$ see \cite{MR2299456,MR2466966,MR2470818,RBK05,MR3624614} for partial results. The second part of the problem remains open.

%For a smooth surface with absolutely integrable Gaussian curvature, its total curvature encodes the global geometric information of the space, e.g. the boundary at infinity, see \cite{SST03}. For example, the total curvature of a convex surface in $\R^3$ describes the apex angle of the cone at infinity of the surface, which is useful to study global geometric and analytic properties of the surface, such as harmonic functions and heat kernels etc., following Colding and Minicozzi \cite{ColdingMinicozzi97JDG} and Xu \cite{Xu14}. In this paper, we study the total curvature of semiplanar graphs with nonnegative curvature.

In this paper, we study the class of planar graphs with nonnegative curvature $\NNG.$ It turns out the class $\NNG$ is much larger than $\PC$ and contains many interesting examples. Among the finite planar graphs in $\NNG$ is a family of so-called fullerenes. A fullerene is a finite cubic planar graph whose faces are either pentagon or hexagon. There are plenty of examples of fullerenes which are important in the real-world applications, to cite a few \cite{KHBCS,MR1668340,MR1441972,BGM12, MR3706860, BuE17b}.  Among the infinite planar graphs in $\NNG$ is a family of all planar tilings with regular polygons as tiles, see e.g. \cite{MR992195,Gal09}. These motivate our investigations of the structure of the class $\NNG.$ For a planar graph $G,$ we denote by $$\Phi(G):=\sum_{x\in V}\Phi(x)$$ the total curvature of $G.$ For a finite planar graph $G,$ Gauss-Bonnet theorem states that $\Phi(G)=2.$ For an infinite planar graph $G\in \NNG,$ the Cohn-Vossen type theorem, see \cite{MR2299456,MR2410938}, yields that
\begin{equation}\label{CVthm}\Phi(G)\leq 1.\end{equation} In \cite{HuaSu17a}, the authors proved that the total curvature for a planar graph with nonnegative curvature is an integral multiple of $\frac{1}{12}.$ %\blue{And in \cite{HuaSu17}, the authors gave a metric rigidity result for nonnegative curvature planar graph with total curvature $=\frac{1}{12}$. This result provide a heuristic idea for constructing
%the planar graph with maximal vertices.} 

For any $G=(V,E,F)\in \NNG,$ we denote by
$$T_G:=\{v\in V: \Phi(x)>0\}$$ the set of vertices with non-vanishing curvature.
For any infinite planar graph in $\NNG$, Chen and Chen \cite{MR2410938,MR2470818} obtained
the interesting result that $T_G$ is a finite set. By Alexandrov's embedding theorem \cite{MR2127379}, the polyhedral surface $S(G)$ can be isometrically embedded into $\R^3$ as a boundary of a noncompact convex polyhedron. The set $T_G$ serves as the set of the vertices/corners of the convex polyhedron, so that much geometric information of the polyhedron is contained in $T_G.$  We are interested in the structure of the set $T_G.$

Analogous to the prisms and antiprisms in $\PC,$
we define some similar families of planar graphs in $\NNG.$
\begin{definition}\label{def:prism} We call a planar graph $G=(V,E,F)\in \NNG$ a prism-like graph if either \begin{enumerate}\item $G$ is an infinite graph and $D_G\geq 43,$ where $D_G$ is defined in Equation \eqref{eq:max}, or 
\item $G$ is a finite graph and there are at least two faces with degree at least $43.$
\end{enumerate}
\end{definition}
The name of ``prism-like" graph is chosen since the structure of these graphs is simple in some sense, analogous to a prism or an antiprism, and can be completely determined, see Theorem~\ref{thm:prminf} and Theorem~\ref{thm:finiteprism}. One may ask the following problem analogous to that of DeVos and Mohar.
\begin{problem}\label{prob:main} What are the numbers $$K_{\SP^2}:=\max_{\mathrm{finite}\  G}{\sharp T_G},\quad K_{\R^2}:=\max_{\mathrm{infinite}\ G}{\sharp T_G},$$ where the maxima are taken over finite and infinite graphs in $\NNG$ which are not prism-like graphs respectively?
\end{problem}

%For a planar graph $G,$ we denote by $$\Phi(G):=\sum_{x\in V}\Phi(x)$$ the total curvature of $G$ whenever the summation converges absolutely. 
%For an infinite planar graph $G,$ the Cohn-Vossen type theorem, see \cite[Theorem~1.3]{MR2299456} or \cite[Theorem~1.6]{MR2410938}, yields that
%\begin{equation}\label{CVthm}\Phi(G)\leq\chi(S(G)),\end{equation}
%whenever $\sum_{x\in V}\min\{\Phi(x),0\}$ converges.
The second part of the problem was proposed in \cite{MR3547929} and an elementary result was obtained therein, $$K_{\R^2}\leq 1722.$$ In this paper, we give the answer to the second part of the problem.
\begin{theorem}\label{main thm}
$$K_{\R^2}=132.$$
Moreover, a graph in this class attains the maximum if and only if its polyhedral surface contains $12$ disjoint hendecagons.
\end{theorem} On one hand, we give the upper bound $K_{\R^2}\leq 132$ by the discharging method in Section~\ref{section:3}. On the other hand, we construct an example possessing $132$ vertices with non-vanishing curvature as in Figure~\ref{size132}, see the construction in Section~\ref{section:4}.

\begin{figure}[htbp]
\begin{center}
\includegraphics[width=\linewidth]{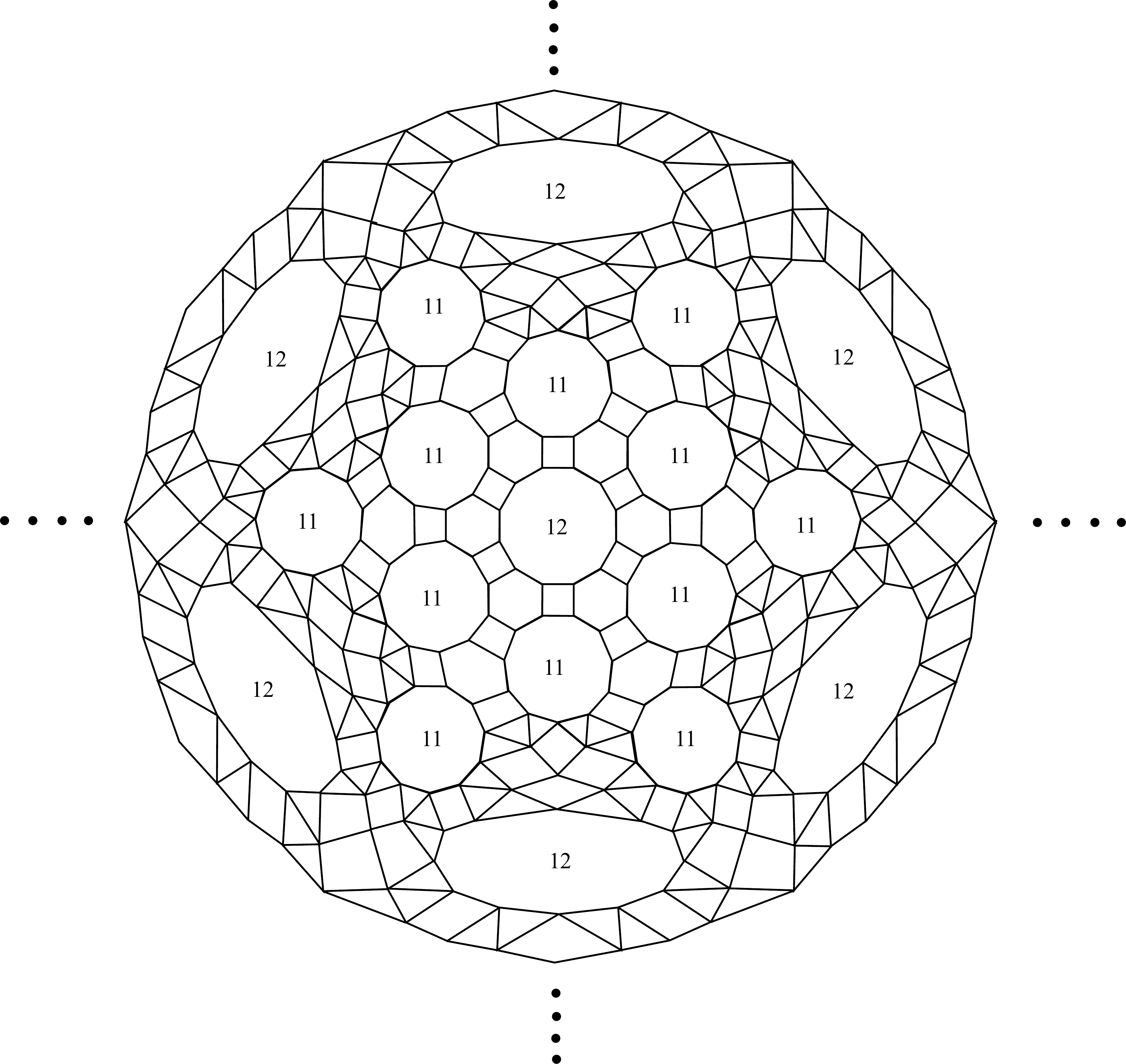}
\caption{\small A planar graph $G\in \NNG$ with $\sharp T_G=132.$}
\label{size132}
\end{center}
\end{figure}

For the first part of Problem~\ref{prob:main}, we give a family of infinitely many examples with arbitrarily large number of vertices of non-vanishing curvature, see Example~\ref{ex:oneface}, which are not prism-like graphs, and hence $$K_{\SP^2}=\infty.$$ By modifying the definition of $K_{\SP^2},$ we get the following result.
\begin{theorem}\label{thm:finite} Let $$\wt{K}_{\SP^2}:=\max_{\mathrm{finite}\  G}{\sharp T_G},$$ where the maximum is taken over finite graphs in $\NNG$ whose maximal facial degree is less than $132.$ Then $$\wt{K}_{\SP^2}=264.$$ Moreover, a graph in this class attains the maximum if and only if its polyhedral surface contains $24$ disjoint hendecagons.
\end{theorem}
The proof of the theorem follows from the same argument as in Theorem~\ref{main thm}, see Section~\ref{section:3} and Section~\ref{section:4}. The upper bound of maximal facial degree, $132,$ in the above theorem is chosen by considering both the discharging argument and concrete examples with many vertices of non-vanishing curvature in Example~\ref{ex:oneface}, see the explanations above that example.

As an application, we may estimate the order of automorphism groups of planar graphs with nonnegative curvature. The automorphism groups of planar graphs have been extensively studied in the literature, to cite a few \cite{MR0296808,MR0371715,MR2410150,MR1739919}. Let $G=(V,E,F)$ be a planar graph. A bijection $R:V\to V$ is called a graph automorphism if it preserves the graph structure of $(V,E).$ A triple $(H_V,H_E,H_F)$ with bijections on $V,$ $E$ and $F$ respectively is called a cellular automorphism of $G=(V,E,F)$ if they preserve the incidence structure of the cell complex $G.$ We denote by $\Aut(G)$ the graph automorphism group of the graph $(V,E),$ and by $\wt{\Aut}(G)$ the cellular automorphism group of the planar graph $G,$ see Section~\ref{sec:auto} for definitions. We prove for any graph $G\in\NNG$ with positive total curvature, the cellular automorphism group is finite, and give the estimate for the order of the group.
\begin{theorem}\label{thm:autoint}
Let $G=(V,E,F)$ be a planar graph with nonnegative combinatorial curvature and $\Phi(G)>0.$ Then we have the following:
\begin{enumerate} \item If $G$ is infinite, then 
$$\sharp \wt{\Aut}(G)\leq \left\{\begin{array}{ll}132!\times 5!, & \mathrm{for}\ D_G\leq 42,\\
2D_G,& \mathrm{for}\ D_G> 42.\end{array}\right.$$%\begin{itemize}\item 
% $Q(G)\leq 132!\times 5!$ for $D_G\leq 42$ and
%\item $Q(G)\leq 2D_G$ for $D_G>42.$ 

%\end{itemize}
\item If $G$ is finite, then %\begin{itemize}\item 
% $Q(G)\leq 264!\times 5!$ for $D_G\leq 42$ and
%\item $Q(G)\leq 4D_G$ for $D_G>42.$ 
%\end{itemize}
$$\sharp \wt{\Aut}(G)\leq \left\{\begin{array}{ll}264!\times 5!, & \mathrm{for}\ D_G\leq 42,\\
4D_G,& \mathrm{for}\ D_G> 42.\end{array}\right.$$
\end{enumerate}
\end{theorem}

Note that Whitney \cite{MR1506961} proved a well-known theorem that any finite 3-connected planar graph, i.e. remaining connected after deleting any two vertices, can be uniquely embedded into $\SP^2$. This has been generalized to infinite graphs by Mohar \cite{MR923264} that a locally finite, 3-connected planar graph, whose faces are bounded by cycles of finite size, has a unique embedding in the plane, see also \cite{MR0384588,MR685062}. These results imply that for any 3-connected planar graph $G=(V,E,F)$ in our setting, 
$$\wt{\Aut}(G)\cong\Aut(G).$$ Hence all results in Theorem~\ref{thm:autoint} apply to the graph automorphism group if the graph is 3-connected.

The paper is organized as follows: In the next section, we recall some basic facts on the combinatorial curvature of planar graphs.
Section~\ref{section:3} is devoted to the upper bound estimates for Theorem~\ref{main thm} and Theorem~\ref{thm:finite}. %In Section~\ref{s:thm1}, we prove Theorem~\ref{mainthm} and give an application, Corollary~\ref{coro:concomp}.
In Section~\ref{section:4}, we construct examples to show the lower bound estimates for Theorem~\ref{main thm} and Theorem~\ref{thm:finite}. The last section contains the proof of Theorem~\ref{thm:autoint}.

\section{Preliminaries}

Let $G=(V,E,F)$ be a planar graph induced by an embedding of a graph $(V,E)$ into $\SP^2$ or $\R^2.$  We only consider the appropriate embedding such that $G$ is a tessellation of $S,$ see the definition in the introduction. Hence $G$ is a finite graph if and only if it embeds into $\SP^2,$ and $G$ is an infinite graph if and only if it embeds into $\R^2.$

We say that a vertex $x$ is incident to an edge $e,$ denoted by $x\prec e,$ (similarly, an edge $e$ is incident to a face $\sigma$, denoted by $e\prec \sigma$; or a vertex $x$ is incident to a face $\sigma,$ denoted by $x\prec \sigma$) if the former is a subset of the closure of the latter. For any face $\sigma,$ we denote by $$\partial\sigma:=\{x\in V: x\prec \sigma\}$$ the vertex boundary of $\sigma.$ Two vertices are called neighbors if there is an edge connecting them. We denote by $\deg(x)$ the degree of a vertex $x,$ i.e. the number of neighbors of a vertex $x,$ and by $\deg(\sigma)$ the degree of a face $\sigma,$ i.e. the number of edges incident to a face $\sigma$ (equivalently, the number of vertices incident to $\sigma$). Two faces $\sigma$ and $\tau$ are called adjacent, denoted by $\sigma\sim\tau,$ if there is an edge incident to both of them, i.e. they share a common edge. Note that by the tessellation properties they share at most one edge.
%Given two edges, we say that one is adjacent to the other if they share a common end-vertex (or end-point).
%The combinatorial distance between two vertices $x$ and $y,$ denote by $d(x,y),$ is defined as the minimal length of walks from $x$ to $y,$ i.e. the minimal number $n$ such that there is $\{x_{i}\}_{i=1}^{n-1}\subset V$ satisfying $x\sim x_1\sim\cdots\sim x_{n-1}\sim y.$ We denote by $B_r(x):=\{y\in V: d(y,x)\leq r\},$ $r\geq 0,$ the ball of radius $r$ centered at the vertex $x.$ 

For a planar graph $G=(V,E,F),$ let $S(G)$ denote the polyhedral surface with piecewise flat metric defined in the introduction. For $S(G),$ it is locally isometric to a flat domain in $\R^2$ near any interior point of an edge or a face, while it might be non-smooth near some vertices. As a metric surface, the generalized Gaussian curvature $K$ of $S(G)$ vanishes at smooth points and can be regarded as a measure concentrated on the isolated singularities, i.e. on vertices. One can show that the mass of the generalized Gaussian curvature at each vertex $x$ is given by
$K(x) = 2\pi-\Sigma_x,$
where $\Sigma_x$ denotes the total angle at $x$ in the metric space $S(G),$ see \cite{MR2127379}. Moreover, by direct computation one has
$K(x) = 2\pi\Phi(x),$ where the combinatorial curvature $\Phi(x)$ is defined in \eqref{def:comb}. Hence one can show that a planar graph $G$ has nonnegative combinatorial curvature if and only if the polyhedral surface $S(G)$ is a generalized convex surface.

In a planar graph, the pattern of a vertex $x$ is defined as a vector $$(\mathrm{deg}(\sigma_1),\mathrm{deg}(\sigma_2),\cdots,\mathrm{deg}(\sigma_{N})),$$ where $N=\deg(x),$ $\{\sigma_i\}_{i=1}^N$ are the faces to which $x$ is incident, and $\mathrm{deg}(\sigma_1)\leq\mathrm{deg}(\sigma_2)\leq\cdots\leq\mathrm{deg}(\sigma_{N}).$ 
%For simplicity, we always write $$x=(\mathrm{deg}(\sigma_1),\mathrm{deg}(\sigma_2),\cdots,\mathrm{deg}(\sigma_{N}))$$ to indicate the pattern of the vertex $x.$

Table \ref{tabl1} is the list of all
possible patterns of a vertex with positive curvature (see
\cite{MR2299456,MR2410938}); Table \ref{tabl2} is the list of all
possible patterns of a vertex with vanishing curvature (see \cite{MR992195,MR2410938}).

\begin{table}
\refstepcounter{table}\label{tabl1}
\begin{tabular}{|lc|lr|}
\hline
Patterns &&$\Phi(x)$&\\
    \hline
 $(3,3,k)$ & $3\leq k$&$1/6+1/k$&\\
 $(3,4,k)$  & $4\leq k$  &$1/12+1/k$&\\
 $(3,5,k)$ & $5\leq k$&$1/30+1/k$&\\
$(3,6,k)$&$6\leq k$&$1/k$&\\
$(3,7,k)$&$7\leq k\leq41$&$1/k-1/42$&\\
$(3,8,k)$&$8\leq k\leq 23$&$1/k-1/24$&\\
$(3,9,k)$&$9\leq k\leq 17$&$1/k-1/18$&\\
$(3,10,k)$&$10\leq k\leq 14$&$1/k-1/15$&\\
$(3,11,k)$&$11\leq k\leq 13$&$1/k-5/66$&\\
$(4,4,k)$&$4\leq k$&$1/k$&\\
$(4,5,k)$&$5\leq k\leq 19$&$1/k-1/20$&\\
$(4,6,k)$&$6\leq k\leq 11$&$1/k-1/12$&\\
$(4,7,k)$&$7\leq k\leq 9$&$1/k-3/28$&\\
$(5,5,k)$&$5\leq k\leq 9$&$1/k-1/10$&\\
$(5,6,k)$&$6\leq k\leq 7$&$1/k-2/15$&\\
$(3,3,3,k)$&$3\leq k$&$1/k$&\\
$(3,3,4,k)$&$4\leq k\leq 11$&$1/k-1/12$&\\
$(3,3,5,k)$&$5\leq k\leq 7$&$1/k-2/15$&\\
$(3,4,4,k)$&$4\leq k\leq 5$&$1/k-1/6$&\\
$(3,3,3,3,k)$&$3\leq k\leq5$&$1/k-1/6$&\\
\hline
\multicolumn{4}{c}{}\\    
\end{tabular}

\textbf{\tablename~\thetable.} The patterns of a vertex with positive curvature.
\end{table}

\begin{table}
\refstepcounter{table}\label{tabl2}
\begin{tabular}{lllll}
\hline
$(3,7,42),$ & $(3,8,24),$ & $(3,9,18),$ & $(3,10,15),$ & $(3,12,12),$\\
$(4,5,20),$ & $(4,6,12),$ & $(4,8,8),$ & $(5,5,10),$ & $(6,6,6),$\\
$(3,3,4,12),$ & $(3,3,6,6),$& $(3,4,4,6),$ & $(4,4,4,4),$ &$(3,3,3,3,6),$\\
$(3,3,3,4,4),$&$(3,3,3,3,3,3).$&&&\\
\hline
\multicolumn{1}{c}{}\\    
\end{tabular}

\textbf{\tablename~\thetable.} The patterns of a vertex with vanishing curvature.
\end{table}

The following lemma is useful for our purposes, see \cite[Lemma~2.5]{MR2410938}.
%\begin{lemma}
%If $0<\Phi(x)<\frac{1}{1722}$, then $x$ is incident to a face $\sigma$ with
%$\mathrm{deg}(\sigma)\geq43$.
%\end{lemma}
\begin{lemma}\label{lemma}
If there is a face $\sigma$ such that $\mathrm{deg}(\sigma)\geq43$ and $\Phi(x)\geq0$ for
any vertex $x$ incident to $\sigma,$ then
$$\sum_{x\in V:x\prec \sigma}\Phi(x)\geq1.$$ %where the summation is taken over all vertices $x$ incident to $\sigma.$
\end{lemma}

For an infinite planar graph $G$ with nonnegative curvature if there is a face of degree at least $43,$ then the graph has rather special
structure, see \cite[Theorem~2.10]{MR3318509}. As in the introduction, we call it an infinite prisim-like graph. 

The following observation is useful for our purposes.
\begin{observation}\label{ob:hex}
For any planar graph $G$ with nonnegative curvature, we divide hexagons into triangles, i.e. each hexagon is divided into $6$ triangles, such that the modified graph possesses no hexagons and still has nonnegative curvature. So that we may assume that $G$ has no hexagons.
\end{observation} We have the following theorem.

\begin{theorem}[\cite{MR3318509}]\label{thm:prminf}  Let $G=(V,E,F)$ be a planar graph with nonnegative curvature and $D_G\geq 43.$ Then there is only one face $\sigma$ of degree at least $43.$ Suppose that there is no hexagonal faces, e.g. by Observation~\ref{ob:hex}. Then the set of faces $F$ consists of $\sigma,$ and triangles and/or squares. Moreover,
$$S(G)=\sigma\cup(\cup_{i=1}^\infty L_i),$$ where $L_i,$ $i\geq 1,$ is a set of faces of the same type (triangle or square) which composite a band, i.e. an annulus, such that $$\min_{\substack{x\in \cup_{\tau\in L_{i}}\partial\tau, \\ y\in \partial\sigma}} d(x,y)=i-1,$$ where $d$ is the graph distance in $(V,E).$ And $S(G)$ is isometric to the boundary of a half flat-cylinder in $\mathbb{R}^3,$ see Figure~\ref{fig:half-cylinder}.

\begin{figure}[htbp]
\begin{center}
\includegraphics[width=\linewidth]{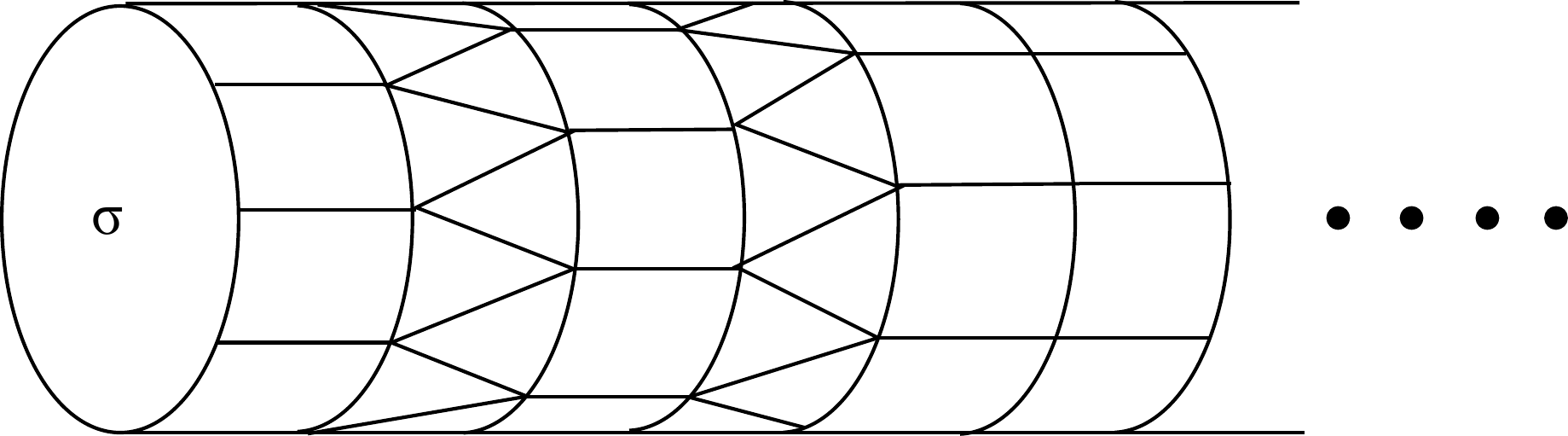}
\caption{\small A half flat-cylinder in $\mathbb{R}^3.$}
\label{fig:half-cylinder}
\end{center}
\end{figure}

%S.G/ D ?? [ S1mD1 Lm, or G has hexagons, i.e. G D P ??1.G0/ where G0 has no hexagons. constructed from a sequence of sets of faces, ??;L1;L2;:::;Lm;:::, where Lm are the sets of faces of the same type (triangle or square), 
\end{theorem}
We collect some properties of an infinite prism-like graph:
\begin{itemize}
\item There is only one face $\sigma$ of degree at least $43,$ and $\deg(\sigma)=D_G$; 
\item $T_G$ consists of all vertices incident to the largest face $\sigma;$ 
\item Any face, which is not $\sigma,$ is either a triangle, a square or a hexagon;
\item The polygonal surface $S(G)$ is isometric to the boundary of a half flat-cylinder in $\mathbb{R}^3.$
\item $\Phi(G)=1.$
 \end{itemize}

Next, we study finite prism-like graphs. Recall that a finite graph $G$ with nonnegative curvature is called prism-like if the number of faces with degree at least $43$ is at least two. 
\begin{theorem}\label{thm:finiteprism} Let $G=(V,E,F)$ be a finite prism-like graph. Then there are exactly two disjoint faces $\sigma_1$ and $\sigma_2$ of same facial degree at least $43.$ Suppose that there is no hexagonal faces, e.g. by Observation~\ref{ob:hex}. Then the set of faces $F$ consists of $\sigma_1$ and $\sigma_2,$ triangles or squares. Moreover,
$$S(G)=\sigma_1\cup(\cup_{i=1}^M L_i)\cup\sigma_2,$$ where $M\geq 1,$ $L_i,$ $1\leq i\leq M,$ is a set of faces of the same type (triangle or square) which composite a band, i.e. an annulus, such that $$\min_{\substack{x\in \cup_{\tau\in L_{i}}\partial\tau, \\ y\in \partial\sigma_1}} d(x,y)=i-1.$$ And $S(G)$ is isometric to the boundary of a cylinder barrel in $\mathbb{R}^3,$ see Figure~\ref{fig:barrel}.

\begin{figure}[htbp]
\begin{center}
\includegraphics[width=0.7\linewidth]{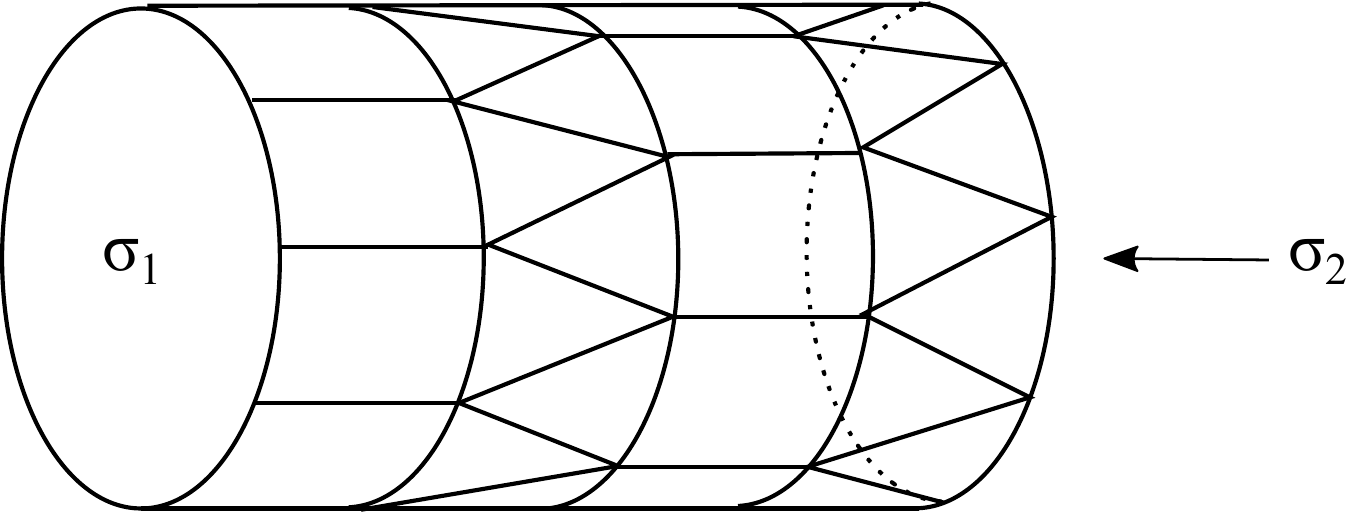}
\caption{\small A cylinder barrel in $\mathbb{R}^3.$}
\label{fig:barrel}
\end{center}
\end{figure}

%S.G/ D ?? [ S1mD1 Lm, or G has hexagons, i.e. G D P ??1.G0/ where G0 has no hexagons. constructed from a sequence of sets of faces, ??;L1;L2;:::;Lm;:::, where Lm are the sets of faces of the same type (triangle or square), 
\end{theorem}
\begin{proof} We adopt the same argument as in the proof of Theorem~\ref{thm:prminf} in \cite{MR3318509}. Since the curvature is nonnegative, one easily sees that the boundaries of any two faces of degree at least $43$ are disjoint. By Lemma~\ref{lemma}, for any face $\sigma$ of degree at least $43,$ $$\sum_{x\in V:x\prec \sigma}\Phi(x)\geq1.$$ Since the graph is finite, by Gauss-Bonnet theorem, $\Phi(G)=2.$ Hence there are exactly two faces of degree at least $43,$ denoted by $\sigma_1$ and $\sigma_2.$ Moreover, we have $$\sum_{x\in V:x\prec \sigma_i}\Phi(x)=1,\quad i=1,2,$$ and $\Phi(x)=0,$ for any $x\in V\setminus(\partial\sigma_1\cup\partial\sigma_2).$ For any $x\in \partial\sigma_1\cup\partial\sigma_2,$ the pattern of $x$ is $(4,4,k_i)$ or $(3,3,3,k_i),$ where $k_i=\mathrm{deg}(\sigma_i)$ for $i=1,2$. Suppose that there is a vertex $x\in \partial\sigma_1$ whose pattern is $(4,4,k_1)$ ($(3,3,3,k_1)$ resp.), then the patterns of all vertices in $\partial\sigma_1$ are $(4,4,k_1)$ ($(3,3,3,k_1)$ resp.). Same results hold for $\partial\sigma_2.$ We denote by $L_1$ the set of all faces incident to a vertex of $\sigma_1,$ and different from $\sigma_1,$ which composite a band. Inductively,  for any $i\geq 1$ define $L_{i+1}$ as the set of faces incident to a vertex of a face in $L_{i},$ which are not in $\sigma_1\cup(\cup_{j=1}^{i-1}L_j).$ Note that $L_i,$ $i\geq 1$, consists of faces of same degree, triangles or squares. Since the graph is finite, there is some $M$ such that $\sigma_2$ is incident to some face in $L_M.$ Considering the properties of faces incident to $\sigma_2$ similarly, we have all faces in $L_M$ are incident to $\sigma_2.$ This yields that $\deg(\sigma_1)=\deg(\sigma_2)$ and proves the theorem.  
\end{proof}
A finite prism-like graph $G$ has the following properties:
\begin{itemize}
\item There are exactly two faces $\sigma_1$ and $\sigma_2$ of degree at least $43;$
\item $\deg(\sigma_1)=\deg(\sigma_2);$  
\item $T_G$ consists of all vertices in $\partial\sigma_1\cup\partial \sigma_2;$ 
\item Any face, which is not $\sigma_1$ and $\sigma_2,$ is either a triangle, a square or a hexagon;
\item The polygonal surface $S(G)$ is isometric to the boundary of a cylinder barrel in $\mathbb{R}^3.$
 \end{itemize}

%Draw a picture for this!!!!!
%Introduce the structure of ``prism-like" for $\NNG.$ here and cite it in the introduction.
%[MR3318509]

\section{Upper bound estimates for the size of $T_G$}\label{section:3}
In this section, we prove the upper bound estimates for the number of vertices in $T_G$ for planar graph $G$ with nonnegative curvature.

\begin{definition}
We say that a vertex $x\in T_G$ is bad if $0<\Phi(x)<\frac{1}{132}$, and good if $\Phi(x)\geq\frac{1}{132}$.
\end{definition}

Let $G\in \NNG$ be either an infinite graph or a finite graph with $D_G< 132,$ which is not a prism-like graph. Then direct computation shows that all patterns of bad vertices of $G$ are given by
$$(3,7,k),\ 32\leq k\leq41, \quad(3,8,k),\ 21\leq k\leq23, \quad(3,9,k),\ 16\leq k\leq 17$$
$$(3,10,14),\quad (3,11,13), \quad(4,5,k),\ 18\leq k\leq19,\quad(4,7,9).$$

%\cite{MR2299456} used this method to show that $C_{\SP^2}\leq 3444,$ which was improved to $C_{\SP^2}\leq 380$ by Oh \cite{MR3624614}. By a refined argument, Ghidelli \cite{Gh17}
Our main tool to prove the results is the discharging method, see e.g. \cite{MR0543792,MR1441258,MR1897390,MR2299456,MR3624614,Gh17}. The curvature at vertices of a planar graph can be regarded as the charge concentrated on vertices. The discharging method is to redistribute the charge on vertices, via transferring the charge on good vertices to bad vertices, such that the final/terminal charge on involved vertices is uniformly bounded below. In the following, we don't distinguish the charge with the curvature. We need to show that for each bad vertex in $T_G$, one can find some nearby vertices with a fair amount of curvature, from which the bad vertex will receive some amount of curvature. By the discharging process, we get the final curvature uniformly bounded below on $T_G$ by the constant $\frac{1}{132}$, so that it yields the upper bound of the cardinality of $T_G$ by the upper bound of total curvature. Generally speaking, since the discharging method is divided into several steps, one shall check that there remains enough amount of curvature at the vertices which are involved in two or more steps to contribute the curvature. However, in our discharging method, we distribute the curvature at each vertex only once.
 
% \red{Read the proof.}
\begin{proof}[Proof of Theorem \ref{main thm} (Upper bound)]
In the following, we prove the upper bound estimate 
$$K_{\R^2}\leq132.$$ Let $G\in \NNG$ be an infinite graph which is not a prism-like graph. {We will introduce a discharging process to distribute some amount of the curvature at good vertices to bad vertices via the discharging rules. We denoted by $\Phi$ the curvature at vertices of $G$ and by $\wt{\Phi}$ the final curvature at vertices after carrying out the following discharging rules. We will show that the final curvature satisfies 
$\sum_{x\in V}\wt{\Phi}(x)=\sum_{x\in V}\Phi(x),$ and $$\wt{\Phi}(x)\geq \frac{1}{132},\quad \forall x\in T_G.$$ This will imply the upper bound estimate $\sharp T_G\leq 132$ since the total curvature $\Phi(G)\leq 1.$} 

The proof consists of three steps: In the first step, we consider the cases for bad vertices, $(3,8,k), 21\leq k\leq23$, $(3,9,k), 16\leq k\leq 17$, $(3,10,14)$ and $(3,11,13)$; In the second step, we deal with the cases, $(3,7,k), 32\leq k\leq41$ and $(4,7,9)$; The last step is devoted to the cases, $(4,5,k), 18\leq k\leq19$.

{\bf Step 1.} In this step, we divide it into the following three cases:
\begin{description}
\item[Case 1.1] There is at least one bad vertex on some $k$-gon with $21\leq k\leq23$. In this case, the pattern of those bad vertices must be $(3,8,k).$ Fix a bad vertex $a$ of the $k$-gon, see Figure \ref{3-8-k}. Then we {will show} that the vertices $x,y$ of the octagon in Figure \ref{3-8-k} and the good vertices on the $k$-gon have enough curvature to be distributed to {all} bad vertices of $k$-gon such that the final curvature {$\wt{\Phi}$} at vertices involved is greater than $\frac{1}{132}.$ In fact, all the possible patterns of $x$ and $y$
are \begin{equation}\label{eq:c2}(3,m,8), 3\leq m\leq 8, (3,3,3,8)\ \mathrm{and}\ (3,3,4,8).\end{equation}

\begin{figure}[htbp]
\begin{center}
\includegraphics[width=0.4\linewidth]{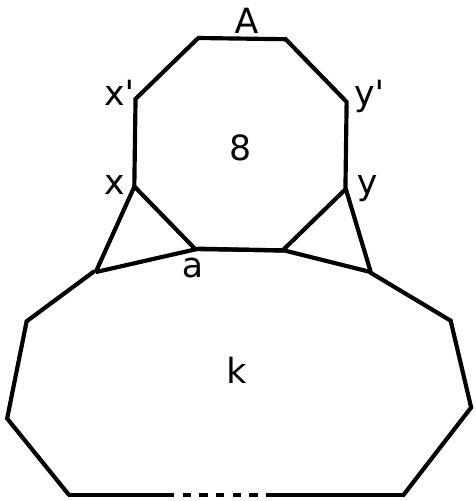}
\caption{\small The case for bad vertices of the pattern $(3,8,k)$ for $21\leq k\leq23$.}
\label{3-8-k}
\end{center}
\end{figure}

Note that the curvature of vertices on the $k$-gon is at least $\frac{1}{k}-\frac{1}{24}$, and for $21\leq k\leq23$, 
\begin{equation}\label{eq:c1}2\times\frac{13}{176}+k\left(\frac{1}{k}-\frac{1}{24}\right)\geq\frac{k+2}{132}.\end{equation} By the above equation \eqref{eq:c1}, if we can find sets of good vertices on the $k$-gon, $A_x$ and $A_y,$ satisfying that 
\begin{itemize}
\item $A_x\cap A_y=\emptyset,$
\item $\Phi(x)+\sum\limits_{z\in A_x}\Phi(z)-\sharp A_x\cdot\left(\frac{1}{k}-\frac{1}{24}\right)>\frac{13}{176}$ and
\item $\Phi(y)+\sum\limits_{z\in A_y}\Phi(z)-\sharp A_y\cdot\left(\frac{1}{k}-\frac{1}{24}\right)>\frac{13}{176},$
\end{itemize}
then $x$, $y$, $A_x$ and $A_y$ have enough curvature to be distributed to all bad vertices of the $k$-gon. {In the following, we try to find the sets $A_x$ and $A_y$ case by case.}

Note that all possible patterns {listed in \eqref{eq:c2}} except $(3,3,4,8)$ have curvature at least $\frac{1}{12}$ which is strictly greater than $\frac{13}{176}$. 
{First, we consider the pattern of the vertex $x.$ It splits into the following two subcases:}
\begin{description}
\item[Subcase 1.1.1] $x$ is not of the pattern $(3,3,4,8)$. We take $A_x=\emptyset$. 
\item[Subcase 1.1.2] $x$ is of the pattern $(3,3,4,8)$. In this situation, we have three cases, as depicted in Figure \ref{3-8-possible}. In either case, $\Phi(z)\geq\frac{1}{k},$ i.e. $z$ has a fair amount of curvature, and
$$\Phi(x)+\Phi(z)-\left(\frac{1}{k}-\frac{1}{24}\right)\geq\frac{1}{12}>\frac{13}{176}.$$
We take $A_x=\{z\}$.
\begin{figure}[tb]
%\begin{minipage}[b]{0.3\textwidth}
%\centering
%\includegraphics[width=3cm,height=3cm]{Case3-1-v2.pdf}
%(a)
%\end{minipage}
\begin{minipage}[b]{0.3\textwidth}
\centering
\includegraphics[width=3.5cm,height=3.5cm]{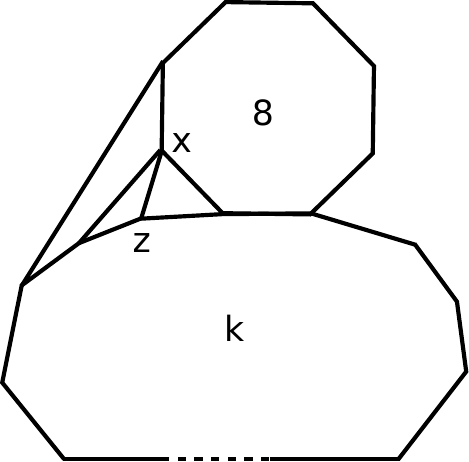}
(a)%% label for entire figure
\end{minipage}
\begin{minipage}[b]{0.3\textwidth}
\centering
\includegraphics[width=3.5cm,height=3.5cm]{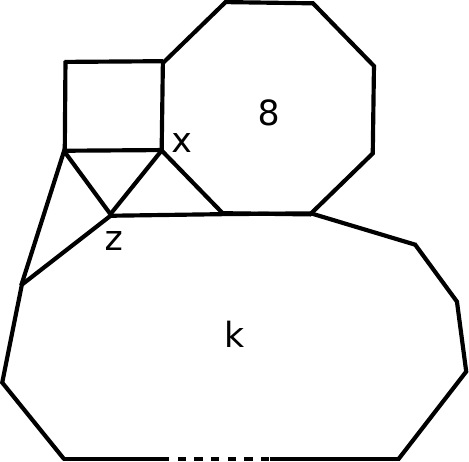}
(b)%% label for entire figure
\end{minipage}
\begin{minipage}[b]{0.3\textwidth}
\centering
\includegraphics[width=3.5cm,height=3.5cm]{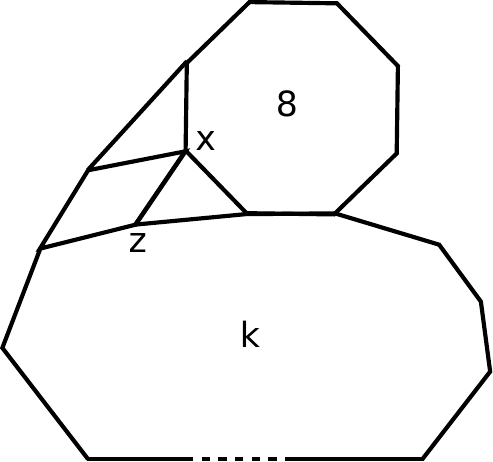}
(c)%% label for entire figure
\end{minipage}
\caption{\small Three situations in Subcase 1.1.2.}
 \label{3-8-possible}
\end{figure}
\end{description}

Considering the pattern of the vertex $y,$ we may choose the set $A_y$ similarly as $A_x$ above. It is easy to check that $A_x$ and $A_y$ satisfy the conditions as required. This finishes the proof for this case.
%then for $21\leq k\leq23$,
%$$\Phi(x)+\Phi(y)+\sum_{z\in k}\Phi(z)\geq\frac{1}{12}+\frac{1}{12}+k\left(\frac{1}{k}-\frac{1}{24}\right)>\frac{k+2}{132}.$$
%This means that $x,y$ have enough curvature to be distributed to all bad vertices of the $k$-gon.

%
%Hence $x$ \blue{and $z$ (similarly for $y$)} also have enough curvature to be distributed to all bad vertices of the $k$-gon.
\begin{remark}
Note that the only edge of the octagon which could be incident to another $l$-gon with $21\leq l\leq23$ is the edge $A,$ see Figure \ref{3-8-k}. Hence we can distribute the curvature of vertices $x'$ and $y'$ to bad vertices on the $l$-gon. This means that the curvature at every vertex on the octagon is transferred to other bad vertices in the discharging process no more than once.
\end{remark}

\item[Case 1.2]  There is at least one bad vertex on some $k$-gon with $16\leq k\leq17$. In this case, the pattern of bad vertices must be $(3,9,k).$ 
The situation is similar to Case 1.1. We can argue verbatim as above to conclude the results, and hence omit the proof here.

\item[Case 1.3]  There is at least one bad vertex on some $14$-gon. In this case, the pattern of bad vertices must be $(3,10,14).$ 
We can use a similar argument as in Case 1.1 and omit the proof here.

\item[Case 1.4] There is at least one bad vertex on some tridecagon, denoted by $\sigma$. In this case, the pattern of bad vertices must be $(3,11,13),$ see Figure \ref{3-11-13}. Note that all possible patterns of $x$ and $y$
are 
$$(3,m,11), 3\leq m\leq 11, (3,3,3,11)\ \mathrm{and}\ (3,3,4,11).$$

\begin{figure}[htbp]
\begin{center}
\includegraphics[width=0.4\linewidth]{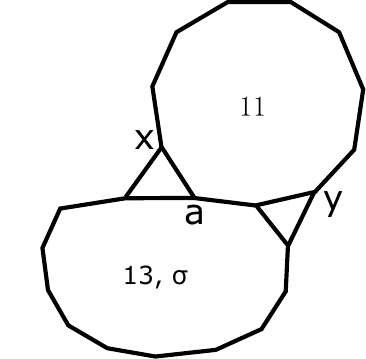}
\caption{\small The case for bad vertices of the pattern $(3,11,13)$.}
\label{3-11-13}
\end{center}
\end{figure}

Note that the curvature of vertices on the tridecagon $\sigma$ is at least $\frac{1}{858}$ and the following holds,
$$2\times\frac{13}{264}+13\times\frac{1}{858}{=}\frac{13+2}{132}.$$
Suppose that we can find sets of good vertices, $A_x,$ $A_y,$ $B_x$ and $B_y,$ satisfying that
%such that $A_x$ and $A_y$ are on the tridecagon, and $B_x$ and $B_y$ good vertices on  which are not on the tridecagon and not contain $x,y$ such that 
\begin{itemize}
\item {$(A_x\cup A_y)\subset \partial \sigma, (B_x\cup B_y)\cap (\partial\sigma\cup \{x,y\})=\emptyset,$}
\item $(A_x\cup B_x)\cap(A_y\cup B_y)=\emptyset,$
\item $\Phi(x)+\sum\limits_{z\in A_x}\Phi(z)-\frac{\sharp A_x}{858}+\sum\limits_{z\in B_x}\Phi(z)-\frac{\sharp B_x}{132}>\frac{13}{264}$ and
\item $\Phi(y)+\sum\limits_{z\in A_y}\Phi(z)-\frac{\sharp A_y}{858}+\sum\limits_{z\in B_y}\Phi(z)-\frac{\sharp B_y}{132}>\frac{13}{264}.$
\end{itemize}
Then $x$, $y$, $A_x$, $A_y$, $B_x$ and $B_y$ have enough curvature to be distributed to all bad vertices of the tridecagon.

%%%%%%%%%%%%%%%%%
%%%%%%%%%%%%%%%%%%%%%%%%%%%%%%%%
%For the patterns of $x$ and $y,$ we have the following cases: $(a)$ Both of $x$ and $y$ are not of the pattern $(3,11,11);$ $(b)$ One of $x$ and $y$ is of the pattern $(3,11,11)$ and the other is not; $(c)$ Both of $x$ and $y$ are of the pattern $(3,11,11).$ 

For our purposes, we first consider the case that the pattern of $x$ is not of the pattern $(3,11,11).$ Then we have the following cases, $\mathbf{A}$--$\mathbf{D}$. 
\begin{description}
\item[A] The pattern of $x$ is one of $(3,3,11)$, $(3,4,11)$, $(3,5,11)$, $(3,6,11)$, $(3,7,11)$ and $(3,3,3,11)$. In this subcase, since $\Phi(x)>\frac{13}{264}$,
we take $A_x=\emptyset$ and $B_x=\emptyset$. 
\item[B] $x$ is of the pattern $(3,8,11)$ or $(3,9,11),$ see Figure \ref{3-8-9-11}. In this subcase, noting that $\Phi(x)+\Phi(z)-\frac{1}{858}>\frac{13}{264}$, we can take $A_x=\{z\}$ and $B_x=\emptyset$.
\begin{figure}[htbp]
\begin{center}
\includegraphics[width=0.4\linewidth]{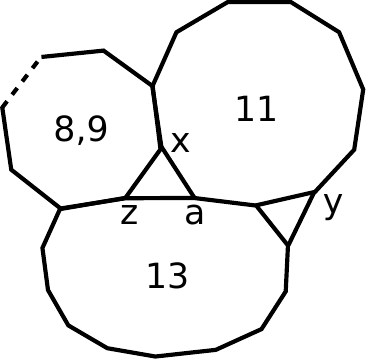}
\caption{\small The case that the pattern of $x$ is $(3,8,11)$ or $(3,9,11)$.}
\label{3-8-9-11}
\end{center}
\end{figure}
\item[C] $x$ is of the pattern $(3,10,11),$ see Figure \ref{3-10-11}. In this subcase, noting that {$\Phi(w')\geq\frac{1}{60}$,} we have
$$\Phi(x)+\Phi(z)+\Phi(w)+\left(\Phi(w')-\frac{1}{132}\right)-\frac{2}{858}>\frac{13}{264}.$$
We can take $A_x=\{z,w\}$ and $B_x=\{w'\}$.
Note that in this situation, the curvature at the vertex $w'$ is not distributed to the vertices of some $14$-gon or other tridecagon, since the edge $A$ cannot be incident to any $14$-gon or tridecagon. That is, we only distribute the curvature of $w'$ to other vertices once. 
\begin{figure}[htbp]
\begin{center}
\includegraphics[width=0.4\linewidth]{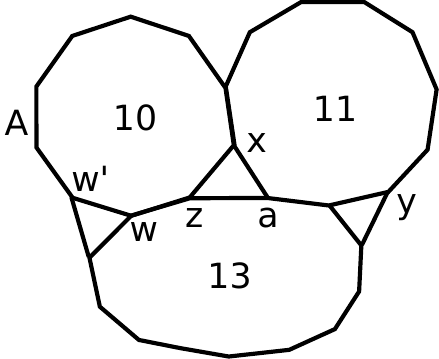}
\caption{\small The case that the pattern of $x$ is $(3,10,11)$.}
\label{3-10-11}
\end{center}
\end{figure}
\item[D] $x$ is of the pattern $(3,3,4,11)$. There are three possible situations, as depicted in Figure \ref{3-11-possible}. In either case, $\Phi(z)\geq\frac{1}{13}$, hence
$$\Phi(x)+\Phi(z)-\left(\frac{1}{13}-\frac{5}{66}\right)\geq\frac{13}{264}.$$
We take $A_x=\{z\}$ and $B_x=\emptyset$.
\begin{figure}[tb]
%\begin{minipage}[b]{0.3\textwidth}
%\centering
%\includegraphics[width=3cm,height=3cm]{Case3-1-v2.pdf}
%(a)
%\end{minipage}
\begin{minipage}[b]{0.3\textwidth}
\centering
\includegraphics[width=3.5cm,height=3.5cm]{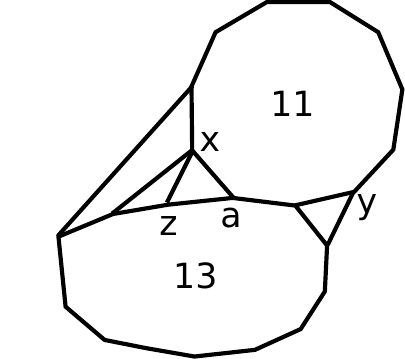}
(a)%% label for entire figure
\end{minipage}
\begin{minipage}[b]{0.3\textwidth}
\centering
\includegraphics[width=3.5cm,height=3.5cm]{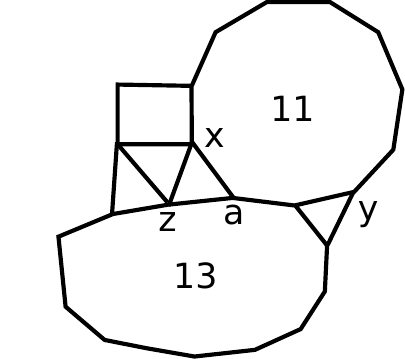}
(b)%% label for entire figure
\end{minipage}
\begin{minipage}[b]{0.3\textwidth}
\centering
\includegraphics[width=3.5cm,height=3.5cm]{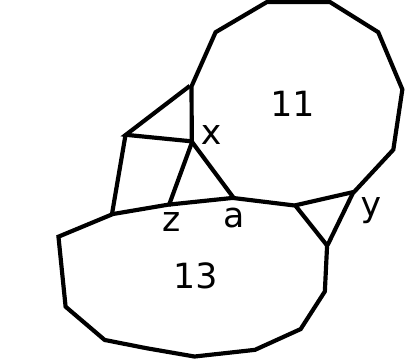}
(c)%% label for entire figure
\end{minipage}
\caption{\small The case that the pattern of $x$ is $(3,3,4,11)$.}
 \label{3-11-possible}
\end{figure}
\end{description}

In case that $y$ is not of the pattern $(3,11,11),$ we can take $A_y$ and $B_y$ according to the above cases, $\mathbf{A}$--$\mathbf{D},$ similarly.

Considering the patterns of $x$ and $y$ together, we have the following subcases for Case 1.4: 
\begin{description}
\item[Subcase 1.4.1] Both of $x$ and $y$ are not of the pattern $(3,11,11).$ In this subcase, we choose $A_x,$ $B_x,$ $A_y$ and $B_y$ according to the cases $\mathbf{A}$--$\mathbf{D}$ mentioned above.
\item[Subcase 1.4.2] One of $x$ and $y$ is of the pattern $(3,11,11)$ and the other is not. Without loss of generality, we may assume that $x$ is of the pattern $(3,11,11)$ and $y$ is not. We can take $A_y$ and $B_y$ according to the cases $\mathbf{A}$--$\mathbf{D}$ above, since the pattern of $y$ is not $(3,11,11).$ Consider the patterns of vertices which have a neighbor on the boundary of the tridecagon, e.g. $x_1,x_2$ etc. as depicted in Figure~\ref{3-11-11}.  If $x_1$ is not of the pattern $(3,11,11)$, then we can take $A_{x_1}$ and $B_{x_1}$ for the vertex $x_1$ according to the cases $\mathbf{A}$--$\mathbf{D}$ above. This yields four sets of good vertices $A_{x_1},$ $B_{x_1},$ $A_{y}$ and $B_{y}$ satisfying the properties as required and proves the result. So that it suffices to consider the case that $x_1$ is of the pattern $(3,11,11).$ If we further consider the vertices $x_2$ and $x_3$ by the same argument as above consecutively, as depicted in Figure~\ref{3-11-11}, then it suffices to consider that $x_2$ and $x_3$ are of the pattern $(3,11,11).$ Now we consider the vertex $x_4.$ If it is not of the pattern $(3,10,11)$ or $(3,11,11)$, then we can take $A_{x_4}$ and $B_{x_4}$ according to the cases $\mathbf{A}$--$\mathbf{D}.$ {Note that in this case {$A_{x_4}=\emptyset$ or $\{c\}$ and }$B_{x_4}=\emptyset$. Hence $A_{x_4}$, $B_{x_4}$, $A_y$ and $B_y$ satisfy the properties as required.} Suppose that $x_4$ is of the pattern $(3,10,11)$ or $(3,11,11)$, then we get that the pattern of $b$ is $(3,3,3,13)$ and $\Phi(b)=\frac{1}{13}$. The vertices $x$, $x_1$ and $b$ have enough curvature to be distributed to all bad vertices on the tridecagon, since
$$\frac{2}{66}+\frac{1}{13}+\frac{12}{858}>\frac{13+2}{132}.$$
%satisfies one of the cases discussed above. We can take $x_1$ in place of $x$ and find $A_{x_1}$ and $B_{x_1}$ which satisfies the condition.
\begin{figure}[htbp]
\begin{center}
\includegraphics[width=0.4\linewidth]{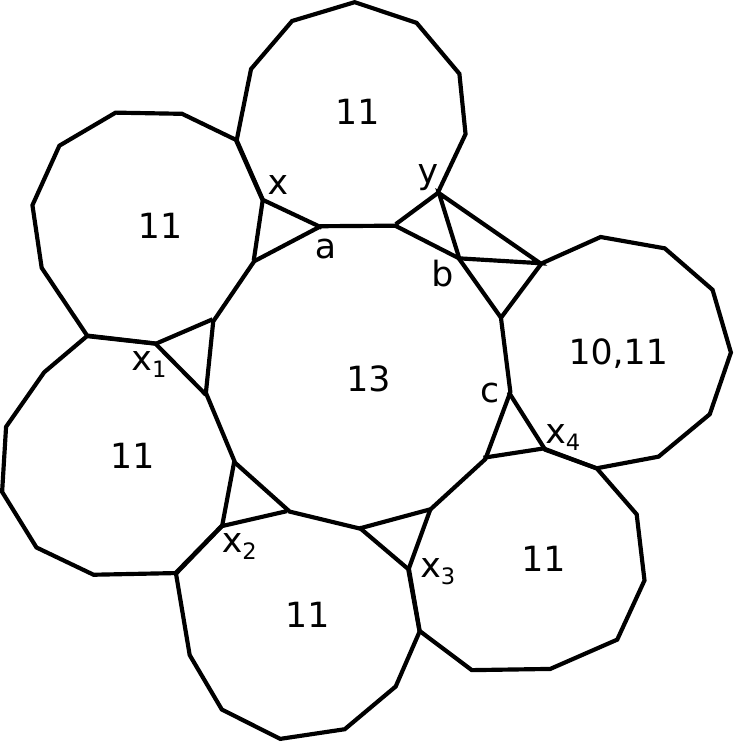}
\caption{\small The case that the pattern of $x$ is $(3,11,11)$.}
\label{3-11-11}
\end{center}
\end{figure}
\item[Subcase 1.4.3] Both of $x$ and $y$ are of the pattern $(3,11,11).$ 
In this subcase, noting that $13$ is an odd number, it is impossible that all vertices, which have a neighbor on the boundary of the tridecagon, are of the pattern $(3,11,11),$ see e.g. Figure~\ref{3-11-11}. Hence, we can find two vertices $x',y',$ which have a neighbor on the boundary of the tridecagon and are incident to a common hendecagon, such that one of them is of the pattern $(3,11,11)$ and the other is not. This reduces this subcase to Subcase 1.4.2 by renaming $x'$ and $y'$ to $x$ and $y.$  %we can assume that $y$ is not of the pattern $(3,11,11),$ see Figure \ref{3-11-11}, that is, $y$ satisfies the subcase 1--4.
\end{description}

\end{description}

{\bf Step 2.} In this step, we consider bad vertices on the heptagons, i.e. those of the patterns $(3,7,k),$ $32\leq k\leq41,$ and $(4,7,9)$. Note that for any two heptagons whose boundaries have non-empty intersection, since the curvature is nonnegative, they must share a common edge, i.e. they are neighbors in the dual graph $G^*.$  We consider connected components of the induced graph on vertices of degree seven in $G^*,$ which corresponds to the heptagons in $G.$ Let $\mathcal{S}$ be such a connected component. We divide it into the following cases.
%Note that both of them contain a heptagon, so that it is possible that there is a heptagon adjacent to a $k$-gon ($32\leq k\leq41$) and an enneagon simultaneously.

%\begin{figure}[tb]
%%\begin{minipage}[b]{0.3\textwidth}
%%\centering
%%\includegraphics[width=3cm,height=3cm]{Case3-1-v2.pdf}
%%(a)
%%\end{minipage}
%\begin{minipage}[b]{0.3\textwidth}
%\centering
%\includegraphics[width=3.5cm,height=3.5cm]{pic6-1-v2.pdf}
%(a)%% label for entire figure
%\end{minipage}
%\begin{minipage}[b]{0.3\textwidth}
%\centering
%\includegraphics[width=3.7cm,height=3.7cm]{pic6-2-v2.pdf}
%(b)%% label for entire figure
%\end{minipage}
%%\begin{minipage}[b]{0.3\textwidth}
%%\centering
%%\includegraphics[width=3.5cm,height=3.5cm]{pic4-3-v2.pdf}
%%(c)%% label for entire figure
%%\end{minipage}
%\caption{\small The case for the pattern $(3,7,k)$.}
% \label{3-7-7}
%\end{figure}
\begin{description}
\item[Case 2.1] $\sharp\mathcal{S}=n\geq 2.$ We denote by $\{\sigma_i\}_{i=1}^n$ the set of vertices in $\mathcal{S},$ where $\sigma_i,$ $1\leq i\leq n$, are heptagons in $G.$ Let $m$ be the number of edges in $\mathcal{S},$ which corresponds to the number of edges shared by the heptagons $\{\sigma_i\}_{i=1}^n.$ We denote by $\{e_i\}_{i=1}^m$ the set of edges shared by the heptagons $\{\sigma_i\}_{i=1}^n$ in $G.$ Hence, the number of vertices in $G$ incident to $\cup_{i=1}^n\sigma_i$ is $7n-2m,$ since any vertex in $G$ is incident to at most two heptagons.  %That is, there are $n$ heptagons $\{\sigma_i\}_{i=1}^n$ such that for any other heptagon $$ in the graph $G,$ there are $\{\sigma_k\}$ 
The two vertices of any edge in $\{e_i\}_{i=1}^m$ are of the pattern $(3,7,7)$ or $(4,7,7),$ hence the curvature at these vertices is at least $\frac{1}{28}$. This means that there are $2m$ vertices with the curvature at least $\frac{1}{28}$. Note that for any heptagon, the faces adjacent to it consists of at most one $k$-gon for $32\leq k\leq42.$ Suppose that there are $s$ heptagons in $\{\sigma_i\}_{i=1}^n$ adjacent to a $k$-gon with $32\leq k\leq 41$ and there are $t$ heptagons in $\{\sigma_i\}_{i=1}^n$ adjacent to a $42$-gon. Clearly $s+t\leq n$. Hence, there are $2s$ vertices with the curvature at least $\frac{1}{1722}$ and $2t$ vertices with vanishing curvature. The curvature of other vertices is at least $\frac{1}{252}$. Since $m\geq n-1>\frac{46}{109}n$ for $n\geq2,$ we have
\begin{eqnarray*}
\frac{2m}{28}+\frac{2s}{1722}+\frac{7n-2m-2(s+t)}{252}
\geq\frac{2m}{28}+\frac{5n-2m}{252}>\frac{7n-2m}{132}.
\end{eqnarray*}
This means that good vertices on these heptagons have enough curvature to be distributed to bad vertices on them.
%deal with the bad vertices are on the $k$-gon with $32\leq k\leq41$.\red{?????The bad vertices are on the $k$-gon with $32\leq k\leq41$.} In this case, let $x$ be any bad vertex of $k$-gon, that is, $x$ is of the pattern $(3,7,k)$. Consider the vertex $y,$ a neighbor of $x,$ such that the triangle and the $k$-gon share the common edge $\{x,y\}.$ 

\begin{figure}[tb]
\begin{minipage}[b]{0.3\textwidth}
\centering
\includegraphics[width=3.5cm,height=3cm]{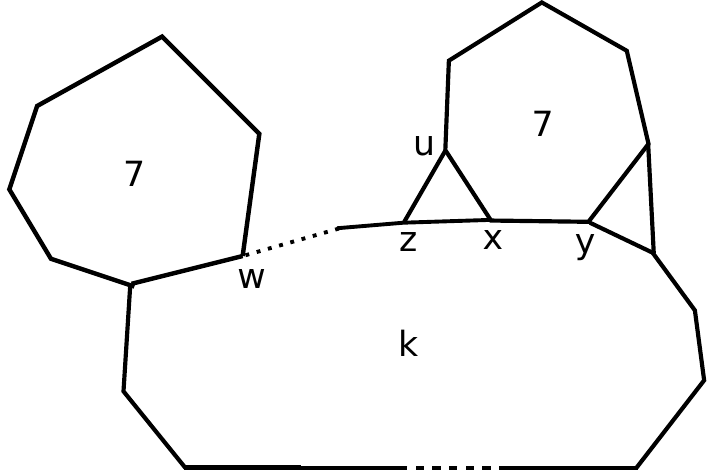}
(a)
\end{minipage}
\begin{minipage}[b]{0.3\textwidth}
\centering
\includegraphics[width=3.3cm,height=3.5cm]{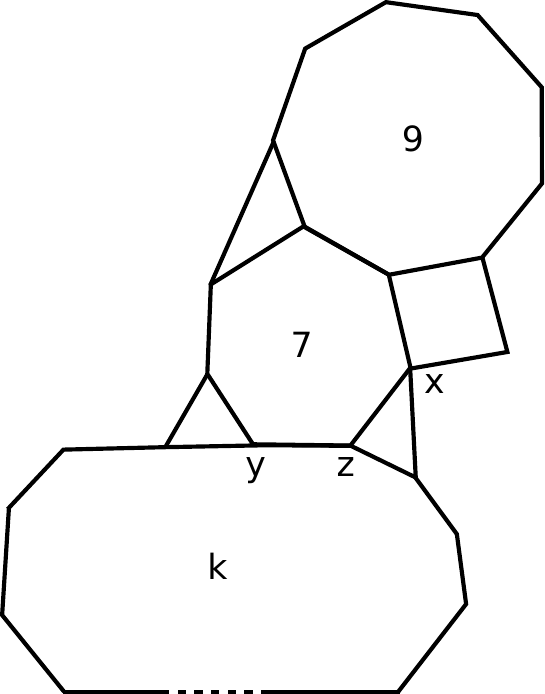}
(b)%% label for entire figure
\end{minipage}
\begin{minipage}[b]{0.3\textwidth}
\centering
\includegraphics[width=3.3cm,height=3.7cm]{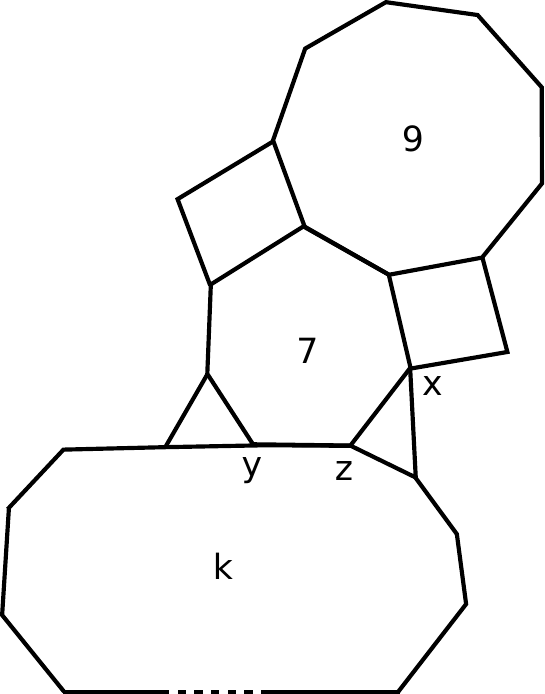}
(c)%% label for entire figure
\end{minipage}
%\begin{minipage}[b]{0.3\textwidth}
%\centering
%\includegraphics[width=3.5cm,height=3.5cm]{pic4-3-v2.pdf}
%(c)%% label for entire figure
%\end{minipage}
\caption{\small The heptagon is adjacent to a $k$-gon for $32\leq k\leq41$.}
 \label{4-7-9}
\end{figure}

\item[Case 2.2] $\sharp\mathcal{S}=1.$ For the heptagon in $\mathcal{S},$ the intersection of its boundary and the boundary of any other heptagon in $G$ is empty. We consider this isolated heptagon and divide it into the following three subcases:

\begin{description}
\item [Subcase 2.2.1] The heptagon is adjacent to a $k$-gon for $32\leq k\leq41$, and there is no vertices on the heptagon of the pattern $(4,7,9),$ see Figure \ref{4-7-9}(a).
In this subcase, all possible patterns of vertices on the $k$-gon are $(3,3,k)$, $(3,4,k)$, $(3,5,k)$, $(3,6,k)$, $(3,7,k)$, $(4,4,k)$ and $(3,3,3,k).$ Except the pattern $(3,7,k)$, all patterns have the curvature at least $\frac{1}{k}$. 
%\Hmm{Where is $u?$} 
Note that the pattern of $z$ cannot be $(3,7,k).$ Otherwise, the vertex $u,$ as depicted in Figure~\ref{4-7-9}(a), would be of pattern $(3,7,7)$ which will imply that two heptagons are adjacent to each other and yields a contradiction to the assumption. Hence $\Phi(z)\geq\frac{1}{k}$. Denote by $w$ the first vertex on the boundary of the $k$-gon in the direction from $y$ to $x$ which is of the pattern $(3,7,k)$. If the heptagon which contains $w$ belongs to Case 2.1, i.e. it is incident to another heptagon, then we distribute the curvature at $z$ only to $x$. Otherwise, we distribute the curvature at $z$ to both $x$ and $w$. For the vertex $y,$ we can use a similar argument to derive the result.
\item [Subcase 2.2.2] The heptagon is adjacent to a $k$-gon for $32\leq k\leq41$, and there exists at least one vertex on the heptagon of the pattern $(4,7,9)$. In this subcase,  there are two situations, see Figure \ref{4-7-9}(b, c). In either case, note that the pattern of $x$ must be $(3,3,4,7)$ and $\Phi(x)=\frac{5}{84}$, which is enough to be distributed to bad vertices of the pattern $(4,7,9)$, since there are at most two such bad vertices on the heptagon. And bad vertices of the pattern $(3,7,k),$ i.e. the vertices $y,z$ in Figure \ref{4-7-9} (b, c), can be treated in the same way as in Subcase 2.2.1. 

\item [Subcase 2.2.3] The heptagon is not adjacent to any $k$-gon for $32\leq k\leq41.$ That is, all bad vertices are of the pattern $(4,7,9)$. Fix a vertex $x$ on the heptagon of the pattern $(4,7,9),$ see Figure \ref{4-7-9-1}. For the vertex $y,$ we have the following possible situations:
\begin{description}
\item[{Subcase 2.2.3(A)}] If the pattern of $y$ is $(3,7,9)$ and $z$ is not of the pattern $(3,7,l)$ for $16\leq l\leq 17$, then the curvature at $y$ has not been distributed to other vertices before. And by the fact that $\Phi(y)=\frac{11}{126}>\frac{7}{132}$, $y$ has enough curvature to be distributed to all bad vertices on the heptagon.

If the pattern of $y$ is $(3,7,9)$ and $z$ is of the pattern $(3,7,l)$ for $16\leq l\leq 17,$ see Figure \ref{4-7-9-1}(a), then the curvature at $y$ might have been distributed to bad vertices of the $l$-gon in Step 1. However, since $z$ and $w,$ as depicted in Figure~\ref{4-7-9-1}(a), must be of the pattern $(3,7,l)$ and $\min\{\Phi(z),\Phi(w)\}\geq\frac{25}{714}$, they have enough curvature to be distributed to all bad vertices of the heptagon.

\item[{Subcase 2.2.3(B)}] If the pattern of $y$ is $(4,7,9)$, then we consider the vertex $z,$ see Figure \ref{4-7-9-1}(b). By the assumption, $z$ cannot be of pattern $(4,7,7)$. If $z$ is not of the pattern $(4,7,8)$ or $(4,7,9)$, then $\Phi(z)\geq\frac{5}{84}>\frac{7}{132}$, hence $z$ has enough curvature to be distributed to bad vertices on the heptagon. 
If the pattern of $z$ is $(4,7,8)$ ($(4,7,9)$ resp.) and the pattern of $w$ is $(3,7,8)$ ($(3,7,9)$ resp.), we can get the conclusion by a similar argument as in the situation {Subcase 2.2.3(A)} above where $y$ is of the pattern $(3,7,9)$.
If the pattern of $w$ is $(4,7,8),$ which implies that the pattern of $z$ must be $(4,7,8)$,  then we have 
$\Phi(z)+\Phi(w)=\frac{1}{28}.$ They have enough curvature to be distributed to bad vertices of the heptagon, since
$$\frac{1}{28}+\frac{5}{252}>\frac{7}{132}.$$ If the pattern of $w$ is $(4,7,9),$ then the pattern of $z$ must be $(4,7,9)$. Similarly, we consider $s$ and $t$, if $s$ or $t$ is not $(4,7,9)$, the conclusion is obvious. If all of them are $(4,7,9)$, then $u$ is of the pattern $(4,4,7)$ which has enough curvature to distribute.
\end{description}
\begin{figure}[tb]
%\begin{minipage}[b]{0.3\textwidth}
%\centering
%\includegraphics[width=3cm,height=3cm]{Case3-1-v2.pdf}
%(a)
%\end{minipage}
\begin{minipage}[b]{0.3\textwidth}
\centering
\includegraphics[width=3.5cm,height=3.5cm]{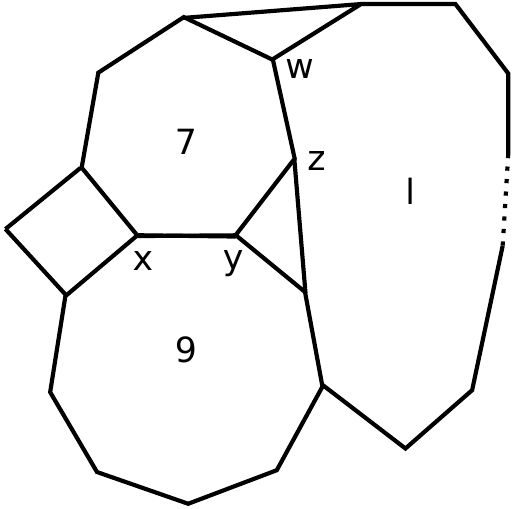}
(a)%% label for entire figure
\end{minipage}
\begin{minipage}[b]{0.3\textwidth}
\centering
\includegraphics[width=3.7cm,height=3.7cm]{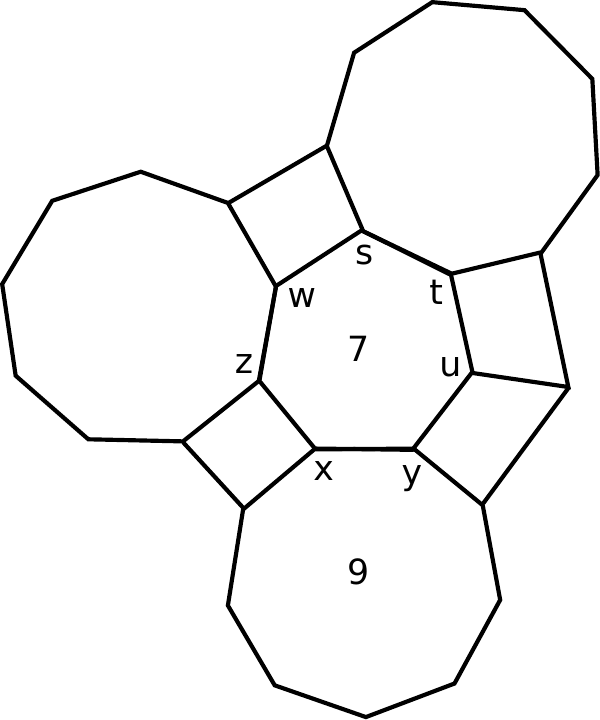}
(b)%% label for entire figure
\end{minipage}
%\begin{minipage}[b]{0.3\textwidth}
%\centering
%\includegraphics[width=3.5cm,height=3.5cm]{pic4-3-v2.pdf}
%(c)%% label for entire figure
%\end{minipage}
\caption{\small The heptagon is not adjacent to any $k$-gon for $32\leq k\leq41$.}
 \label{4-7-9-1}
\end{figure}
\end{description}
\end{description}

{\bf Step 3.} We consider the last case in which bad vertices are on a $k$-gon with $18\leq k\leq19.$ That is, bad vertices are of the pattern $(4,5,k)$ for $18\leq k\leq19$. We divide it into the following cases.
\begin{description}
\item[Case 3.1] There exists a vertex {$z$ on the boundary of} $k$-gon with $18\leq k\leq19$ which is of the pattern $(3,4,k),$ see Figure \ref{3-4-5-k}(a). Note that $\Phi(z)=\frac{1}{12}+\frac{1}{k}.$ 

For $k=18,$ the curvature of a vertex with positive curvature on the $18$-gon is at least $\frac{1}{180}.$ {Let $l$ be the number of vertices with positive curvature on the $18$-gon, $1\leq l\leq 18.$ }Then
$$\Phi(z)+\frac{l-1}{180}>\frac{l}{132}.$$
Hence $z$ has enough curvature to be distributed to all bad vertices on the $18$-gon.

For $k=19$, the curvature on vertices of the $19$-gon is at least $\frac{1}{380}$, which yields that 
$$\Phi(z)+\frac{18}{380}>\frac{19}{132}.$$
Hence $z$ also has enough curvature to be distributed to all bad vertices on the $19$-gon.
\item[Case 3.2] There exists a vertex {$z$ on the boundary of} $k$-gon with $18\leq k\leq19$ which is of the pattern $(3,5,k),$ see Figure \ref{3-4-5-k}(b). Note that $\Phi(z)=\frac{1}{30}+\frac{1}{k}$.
Hence, for the case $k=18$ it is similar to Case 1 that
$$\Phi(z)+\frac{l-1}{180}>\frac{l}{132},$$ {where $l$ is the number of vertices with positive curvature on the $18$-gon.}
For the case of $k=19,$ since $\Phi(w)+\Phi(w')\geq2\left(\frac{1}{k}-\frac{1}{24}\right)$, we have
$$\frac{1}{30}+\frac{1}{k}+2\left(\frac{1}{k}-\frac{1}{24}\right)+(k-3)\left(\frac{1}{k}-\frac{1}{20}\right)>\frac{k}{132}.$$
Hence, {good vertices on the $k$-gon} have enough curvature to be distributed to bad vertices on the $k$-gon.

\begin{figure}[tb]
%\begin{minipage}[b]{0.3\textwidth}
%\centering
%\includegraphics[width=3cm,height=3cm]{Case3-1-v2.pdf}
%(a)
%\end{minipage}
\begin{minipage}[b]{0.3\textwidth}
\centering
\includegraphics[width=3.5cm,height=2.4cm]{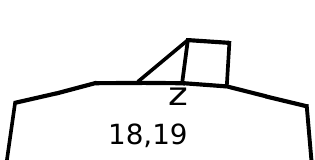}
(a)%% label for entire figure
\end{minipage}
\begin{minipage}[b]{0.3\textwidth}
\centering
\includegraphics[width=3.5cm,height=2.4cm]{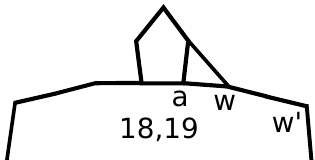}
(b)%% label for entire figure
\end{minipage}
\begin{minipage}[b]{0.3\textwidth}
\centering
\includegraphics[width=3.5cm,height=2.4cm]{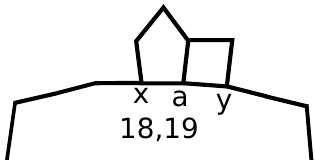}
(c)%% label for entire figure
\end{minipage}
\caption{\small Three cases for Step 3.}
 \label{3-4-5-k}
\end{figure}

\item[Case 3.3]  No vertices on the $k$-gon is of the pattern $(3,4,k)$ or $(3,5,k),$ $18\leq k\leq 19.$ In this case, given a bad vertex $a$ on the $k$-gon, we consider its neighbors $x$ and $y$, see Figure \ref{3-4-5-k} (c). Since there are no vertices on the $k$-gon of the pattern $(3,4,k)$ or $(3,5,k)$, the pattern of $x$ must be $(4,5,k)$ and the pattern of $y$ must be $(4,4,k)$ or $(4,5,k).$ Applying similar arguments to other vertices on the $k$-gon, we conclude that all edges of $k$-gon are incident to squares or pentagons. Now we consider the cases $k=18$ and $k=19$ respectively.
\begin{description}
\item [Subcase 3.3.1] $k=18.$ We observe that if there is a vertex on the boundary of $18$-gon which is of the pattern 
$(4,4,18)$, then this vertex has enough curvature to be distributed to bad vertices on the $18$-gon, since $\frac{1}{18}+\frac{l-1}{180}>\frac{l}{132},$ where $l$ is the number of the vertices on the $18$-gon with positive curvature. Hence it suffices to consider that all patterns of vertices of $18$-gon are $(4,5,18).$

\begin{figure}[htbp]
\begin{center}
\includegraphics[width=0.4\linewidth]{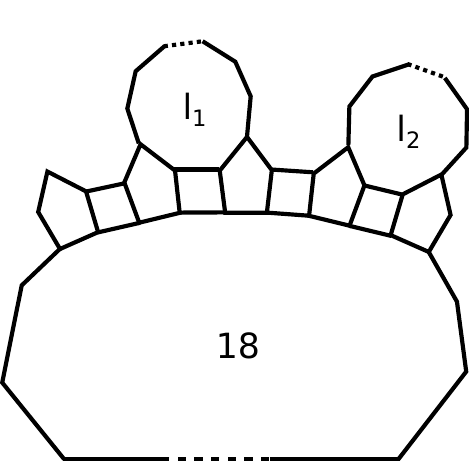}
\caption{\small Two squares, adjacent to the given $18$-gon and  a common pentagon, cannot be in the first class $\mathcal{M}_1$ simultaneously.}
\label{l}
\end{center}
\end{figure}

Note that all patterns which contain a square and a pentagon are $(3,4,5)$, $(4,4,5)$, $(4,5,l), 5\leq l\leq20$, $(3,3,4,5)$, $(3,4,4,5)$. We consider the squares which are adjacent to the given $18$-gon. We divide these squares into two classes: The first class $\mathcal{M}_1$ consists of the squares which are adjacent to another $l$-gon with $7\leq l\leq20;$ The second class $\mathcal{M}_2$ contains all the other squares. A useful
observation is that if two squares, adjacent to the given $18$-gon, are adjacent to a common pentagon, then these  squares cannot be in the first class $\mathcal{M}_1$ simultaneously, see e.g. Figure \ref{l}. Since there are
$9$ squares adjacent to the $18$-gon and $9$ is an odd number, there exist two squares in the second class $\mathcal{M}_2$ adjacent to a common pentagon, see Figure \ref{4-4}(a). Hence possible patterns of $a$, $b$, $c$ and $d,$ depicted in Figure \ref{4-4}(a), are $(3,4,5)$, $(4,4,5)$, $(4,5,5)$, $(4,5,6)$, $(3,3,4,5)$ and
$(3,4,4,5)$. The curvature of a vertex of any pattern above is at least $\frac{1}{30}$, so that they have enough curvature to be distributed to bad vertices on the $18$-gon, since
$$\frac{4}{30}+\frac{1}{180}\times18>\frac{18+4}{132}.$$

\begin{figure}[tb]
%\begin{minipage}[b]{0.3\textwidth}
%\centering
%\includegraphics[width=3cm,height=3cm]{Case3-1-v2.pdf}
%(a)
%\end{minipage}
\begin{minipage}[b]{0.3\textwidth}
\centering
\includegraphics[width=3.7cm,height=2.3cm]{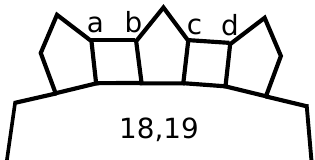}
(a)%% label for entire figure
\end{minipage}
\begin{minipage}[b]{0.3\textwidth}
\centering
\includegraphics[width=3.7cm,height=3.7cm]{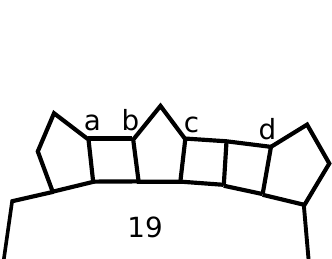}
(b)%% label for entire figure
\end{minipage}
%\begin{minipage}[b]{0.3\textwidth}
%\centering
%\includegraphics[width=3.5cm,height=3.5cm]{pic4-3-v2.pdf}
%(c)%% label for entire figure
%\end{minipage}
\caption{\small There exist two squares in the second class $\mathcal{M}_2$ adjacent to a common pentagon.}
 \label{4-4}
\end{figure}

\item [Subcase 3.3.2] $k=19.$ Note that if there are two vertices on the $19$-gon which are of the pattern 
$(4,4,19)$, then these vertices have enough curvature to be distributed to bad vertices on the $19$-gon, since $\frac{2}{19}+\frac{17}{380}>\frac{19}{132}$. So that it suffices to consider that all patterns of vertices of $19$-gon are $(4,5,19)$ with an exceptional vertex of the pattern $(4,4,19)$. Adopting similar arguments as in Subcase 3.3.1, we have two situations, as depicted in Figure \ref{4-4}(a) and Figure \ref{4-4}(b). In both cases, the vertices $a$,
$b$, $c$ and $d$ have enough curvature to be distributed to bad vertices of the $19$-gon, since
$$\frac{4}{30}+\frac{1}{19}+\frac{1}{380}\times18>\frac{19+4}{132}.$$
\end{description}

\end{description}

Combining all cases above, we distribute the curvature at good vertices to bad vertices such that all vertices in $T_G$ have final curvature $\wt{\Phi}$ uniformly bounded below by $\frac{1}{132}.$ This proves the upper bound $\sharp T_G\leq 132.$ For the equality case, suppose that $G$ has $\sharp T_G=132,$ since all the inequalities in the discharge method are strict, then there are no bad vertices in $T_G$ and all of them have curvature $\frac{1}{132}.$ The vertex pattern of curvature $\frac{1}{132}$ is given by 
 $(3,3,4,11),$ $(4,6,11),$ or $(3,11,12).$ Hence, $G$ has exactly 12 disjoint hendecagons.
\end{proof}

Now we study finite planar graphs with nonnegative curvature. First at all, we show that there are infinitely many finite planar graphs with a large face which are not prism-like graphs. This indicates that the upper bound of maximal facial degree $132$ in Theorem~\ref{thm:finite} is somehow necessary.  
\begin{example}\label{ex:oneface} For any even number $m\geq8,$ there is a finite planar graph $G_m\in \NNG,$ which is not a prism-like graph, such that it has a unique face of degree $m$ and all the other faces are triangles and squares. Moreover, $$\sharp T_G=m+4$$ which yields that $K_{\SP^2}=\infty,$ see Problem~\ref{prob:main} for the definition.
\end{example}
\begin{figure}[htbp]
\begin{center}
\includegraphics[width=0.4\linewidth]{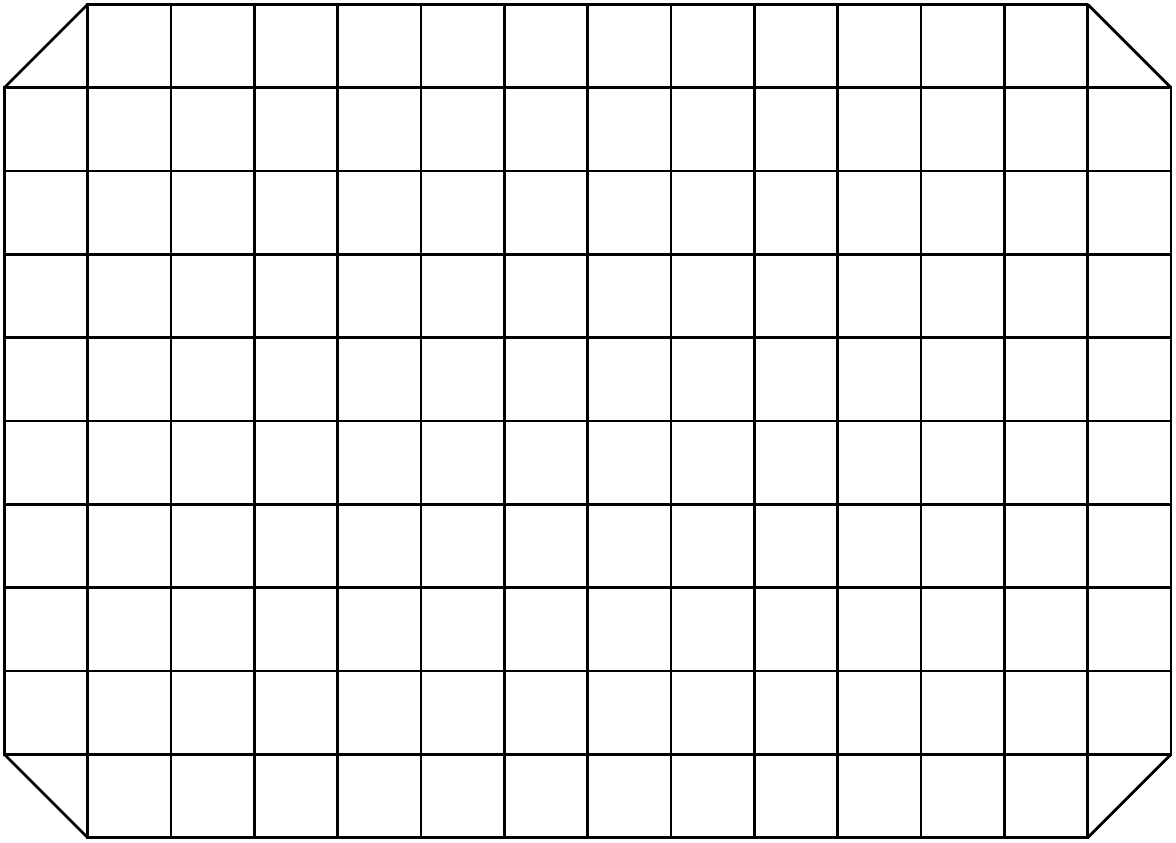}
\caption{\small A finite graph has many vertices with positive curvature, which is not prism-like.}
\label{kkkkk}
\end{center}
\end{figure}
\begin{proof} Write $m=2(a+b+2)$ for some $a\geq 1, b\geq 1.$ We construct a rectangle of side length $a$ and $b,$ consisting of $a\times b$ squares, and attach $2(a+b)$ squares along the sides of the rectangle and $4$ triangles to the four corners. We obtain a convex domain as in Figure \ref{kkkkk} whose boundary consists of $m$ edges. Now glue an $m$-gon along the boundary of the domain, we get a planar graph as desired. 
\end{proof}

\begin{proof}[Proof of Theorem~\ref{thm:finite} (Upper bound)] Let $G$ be a finite planar graph in $\NNG$ with $D_G< 132,$ which is not a prism-like graph. Note that our discharging method in the proof of Theorem~\ref{main thm} is local, one can distribute the curvature to bad vertices such that the modified curvature $\wt{\Phi}\geq \frac{1}{132}$ on $T_G.$ Then by $\Phi(G)=2,$ one gets the upper bound $264.$ For the equality case, one can argue similarly as in the proof of Theorem~\ref{main thm}.
\end{proof}

%Since our discharging method is local. We can use it to the planar graph whose polyhedral surface is homeomorphic to $\mathbb{S}^2$, and get the following corollary:

%\begin{corollary}
%Assume that $G$ is a planar graph whose polyhedral surface is homeomorphic to $\mathbb{S}^2$. Furthermore, we assume that every face is not a big face (i.e. every face has degree less than $132$), then, the maximum number of vertices is 264.
%\end{corollary}

%\begin{remark}
%The condition that every face is not a big face is necessary. For PCC graphs, an amazing result is that there is no PCC graph with maximal face degree greater than $42$. That is, if a planar graph $G$ with positive curvature whose polyhedral surface is homeomorphic to $\mathbb{S}^2$ and there is a face with degree $\geq42$, then $G$ must be a prism or anti-prism (see \cite[Section 13]{Gh17} for details).
%However, for planar graph with nonnegative curvature, there exists such graph whose maximal face degree can be large arbitrary. For example, see Figure \ref{kkkkk} is a graph with squares and triangles
%whose boundary is made up of, for example, $44$ vertices (we can make the number of vertices large arbitrary), we can glue a $44$-gon to construct a planar graph with nonnegative curvature whose polyhedral surface is homeomorphic to $\mathbb{S}^2$ and it is not a prism or anti-prism. 
%\end{remark}

\section{Constructions of large planar graphs with nonnegative curvature}\label{section:4}
In this section, we prove the lower bound estimates for the number of vertices in $T_G$ for a planar graph $G$ with nonnegative curvature in Theorem~\ref{main thm} and Theorem~\ref{thm:finite} by constructing examples. %\blue{In \cite{HuaSu17}, the authors gave the
%metric rigidity result that the nonnegative curvature planar graph with total curvature $=\frac{1}{12}$ must be: there exists one vertex with curvature $\frac{1}{12}$, or there exist 11 vertices (they lie in a hendecagon) with each vertex has curvature $\frac{1}{132}$. Inspired by this result, we can construct a graph with 132 vertices}

In the first part of the section, we consider infinite planar graphs. As shown in the introduction, there is an infinite planar graph $G\in \NNG$ which is not a prism-like graph with $\sharp T_G=132,$ see Figure~\ref{size132}. By this example, we give the lower bound estimate of $K_{\R^2}\geq 132$ in Theorem~\ref{main thm}. However, in a rigorous manner we need to show that the graph in Figure~\ref{size132} can be further extended up to infinity, so as to guarantee that it is a planar tessellation. 

We cut off a central part of Figure~\ref{size132} and denote it by $P_A,$ shown in Figure~\ref{figA}. Then the annular part surrounding it in Figure~\ref{size132} is denoted by $P_B,$ shown in Figure~\ref{figB}, in which we have modified the sizes of faces to obtain a nice-looking annular domain. 
\begin{figure}[htbp]
\begin{center}
\includegraphics[width=0.7\linewidth]{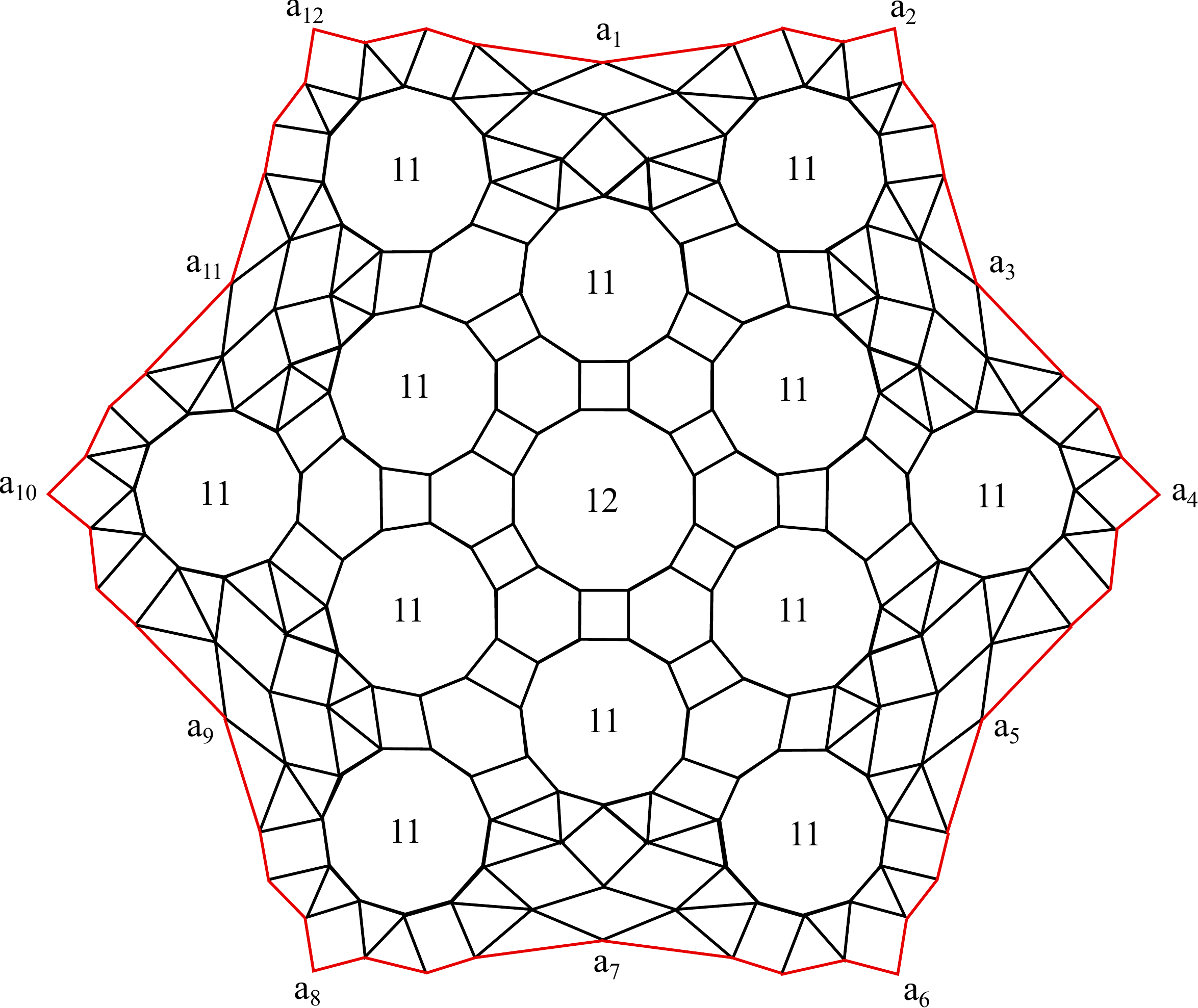}
\caption{\small A central part in Figure~\ref{size132}, denoted by $P_A.$}
\label{figA}
\end{center}
\end{figure}
\begin{figure}[htbp]
\begin{center}
\includegraphics[width=0.7\linewidth]{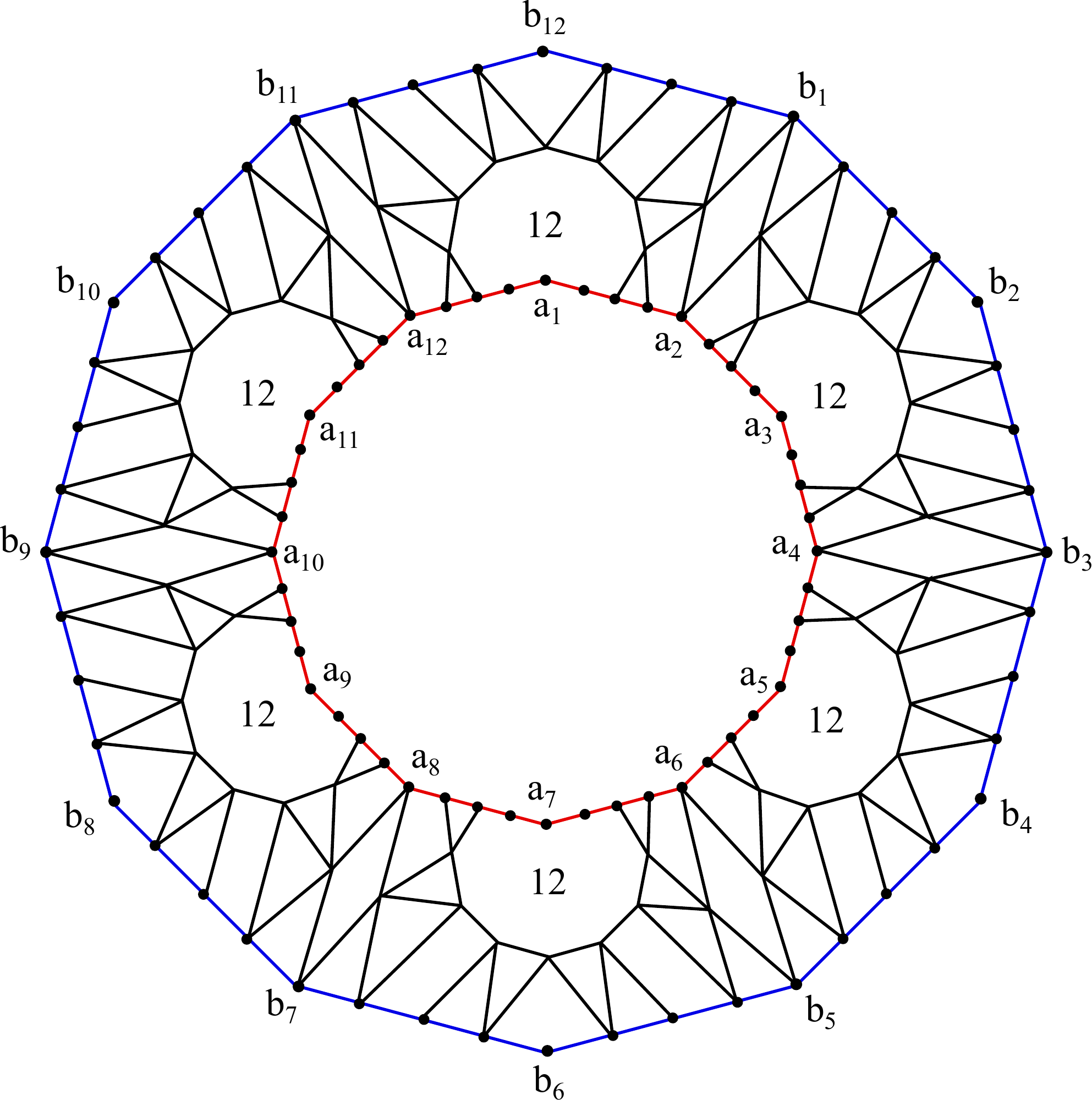}
\caption{\small An annular part in Figure~\ref{size132}, denoted by $P_B.$}
\label{figB}
\end{center}
\end{figure}

\begin{figure}[htbp]
\begin{center}
\includegraphics[width=0.7\linewidth]{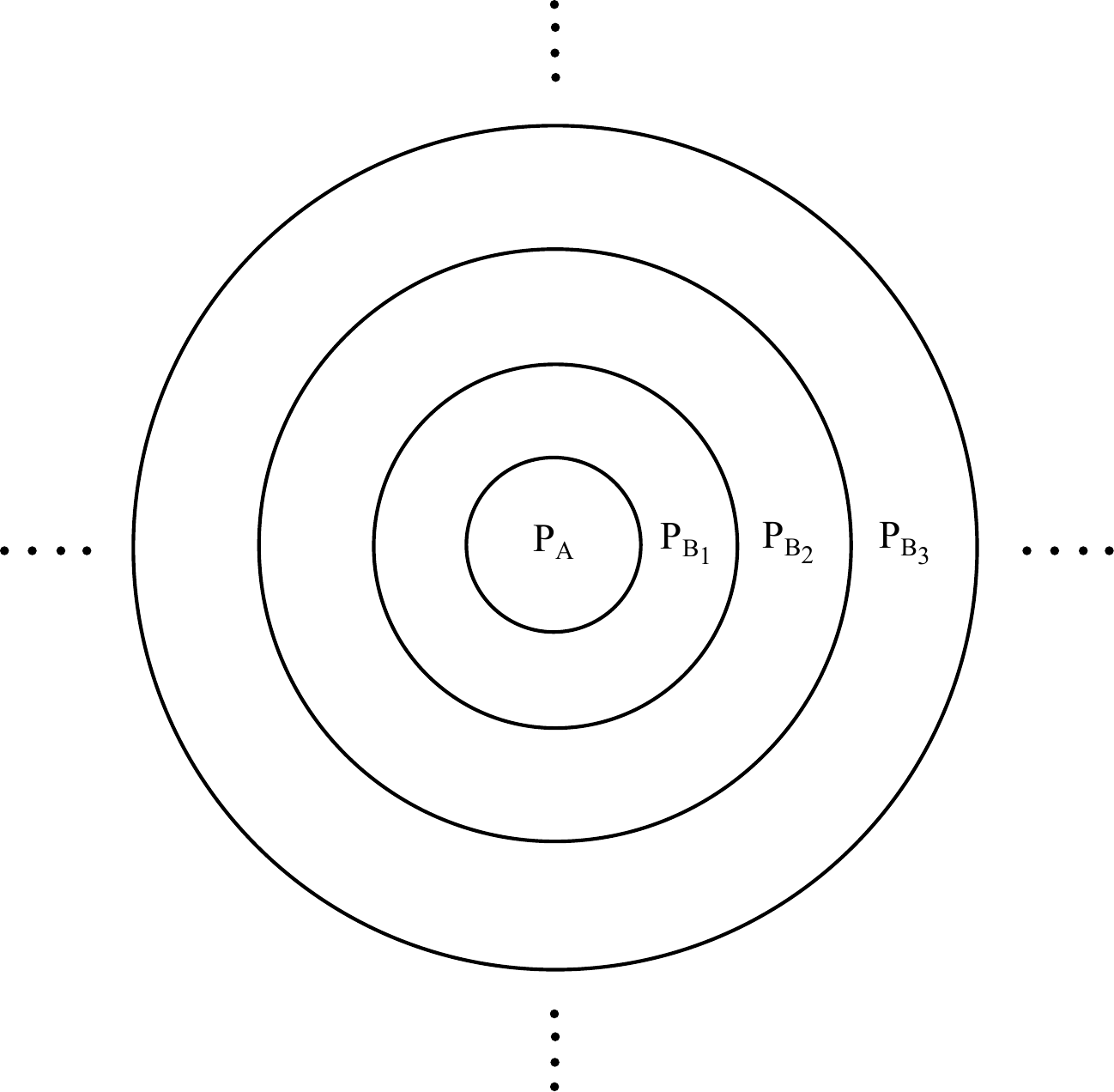}
\caption{\small The gluing process for the construction of Figure~\ref{size132}.}
\label{fig:cons}
\end{center}
\end{figure}
We make infinitely many copies of $P_B,$ denoted by $\{P_{B_i}\}_{i=1}^\infty,$ and glue these pieces together with $P_A$ as follows, see Figure~\ref{fig:cons}:
$$P_A\leftrightarrow P_{B_1}\leftrightarrow P_{B_2}\leftrightarrow P_{B_3}\leftrightarrow \cdots\leftrightarrow P_{B_i}\leftrightarrow \cdots\cdots,$$ where we denote by $G_1\leftrightarrow G_2$ the gluing of two graphs $G_1$ and $G_2$ along the boundaries following the rules:
\begin{enumerate}
\item For $G_1=P_A$ and $G_2=P_{B_1},$ we glue the boundary in $P_A$ with the innermost boundary of $P_{B_1}$ by identifying the vertices $\{a_i\}_{i=1}^{12}$ in both graphs. 
\item For $G_1=P_{B_i}$ and $G_2=P_{B_{i+1}}$ with $i\geq 1,$ we glue the outermost boundary in $P_{B_i}$ with the innermost boundary of $P_{B_{i+1}}$ by identifying the vertices $\{b_i\}_{i=1}^{12}$ in $P_{B_i}$ with the vertices $\{a_i\}_{i=1}^{12}$ in $P_{B_{i+1}}$. 
\end{enumerate}
This constructs the example in Figure~\ref{size132}. One easily sees that this construction has a periodic structure, so that it extends to an infinite planar graph embedded into $\R^2.$

For the second part of the section, we consider finite planar graphs. We will construct a finite planar graph $G\in \NNG$ with $\sharp T_G=264,$ which is not a prism-like graph and has $D_G=12.$  This will give the lower bound estimate of $\wt{K}_{\SP^2}\geq 264$ in Theorem~\ref{thm:finite}. We make two copies of $P_A$ and $P_B$ respectively, denote by $P_{A_1},$ $P_{A_2},$ $P_{B_1}$ and $P_{B_2}.$ We glue them together along the boundaries as follows
$$P_{A_1}\leftrightarrow P_{B_1} \leftrightarrow P_{B_2} \leftrightarrow P_{A_2},$$ where the boundary of $P_{A_1}$ is glued with the innermost boundary of $P_{B_1},$ the outermost boundary of $P_{B_1}$ is glued with the innermost boundary of $P_{B_2}$ by identifying the vertices $\{b_i\}_{i=1}^{12}$ in the former with the vertices $\{a_i\}_{i=1}^{12}$ in the latter, and the outermost boundary of $P_{B_2}$ is glued with the boundary of $P_{A_2}$ by identifying the vertices $\{b_i\}_{i=1}^{12}$ in the former with the vertices $\{a_i\}_{i=1}^{12}$ in the latter. This gives us an example of $\sharp T_G=264.$

\section{Automorphism groups of planar graphs with nonnegative curvature}\label{sec:auto}
In this section, we study automorphism groups of planar graphs with nonnegative curvature.

First, we introduce several definitions of isomorphisms on planar graphs.
\begin{definition} Let $G_1=(V_1,E_1,F_1)$ and $G_2=(V_2,E_2,F_2)$ be two planar graphs. 
\begin{enumerate}
\item $G_1$ and $G_2$ are said to be \emph{graph-isomorphic} if there is a graph isomorphism between $(V_1,E_1)$ and $(V_2,E_2),$ i.e. $R:V_1\to V_2$ such that for any $v,w\in V,$ $v\sim w$ if and only if $R(v)\sim R(w).$
\item $G_1$ and $G_2$ are said to be \emph{cell-isomorphic} if there is a cellular isomorphism $H=(H_V,H_E,H_F)$ between $(V_1,E_1,F_1)$ and $(V_2,E_2,F_2)$ in the sense of cell complexes, i.e. three bijections $H_V:V_1\to V_2,$ $H_E:E_1\to E_2$ and $H_F: F_1\to F_2$ preserving the incidence relations, that is, for any $v\in V, e\in E, \sigma\in F,$ $v\prec e$ if and only if $H_V(v)\prec H_E(e)$ and
$e\prec \sigma$ if and only if $H_E(e)\prec H_F(\sigma);$
\item $G_1$ and $G_2$ are said to be \emph{metric-isomorphic} if there is an isometric map in the sense of metric spaces $L:S(G_1) \to S(G_2),$ such that the restriction map $L$ is cell-isomorphic between $(V_1,E_1,F_1)$ and $(V_2,E_2,F_2).$
\end{enumerate}
\end{definition}
One is ready to see that metric-isomorphic planar graphs are cell-isomorphic, and hence graph-isomorphic.
%For  we call the above isomorphisms automorphisms. 
For a planar graph $G,$ a graph (cellular, metric resp.) isomorphism between $G$ and $G,$ i.e. setting $G_1=G_2=G$ in the above definitions, is called a graph (cellular, metric resp.) automorphism of $G.$ We denote by $\Aut(G),$ ($\wt{\Aut}(G),$ $\mathcal{L}(G)$ resp.) the group of graph (cellular, metric resp.) automorphisms of a planar graph $G.$ By the standard identification, $$\mathcal{L}(G)\leq \wt{\Aut}(G)\leq \Aut(G),$$ where $\leq$ indicates that the former can be embedded as a subgroup of the latter. By our definition of polyhedral surfaces, it is easy to see that
$$\mathcal{L}(G)\cong\wt{\Aut}(G).$$ Moreover, by the results in \cite{MR1506961,MR923264} for a 3-connected planar graph $G,$ any graph automorphism $R$ of $G$ can be uniquely realized as a cellular automorphism $H$ such that $H_V=R,$ which is called the associated cellular automorphism of $R.$ This implies that
$$\mathcal{L}(G) \cong\wt{\Aut}(G)\cong\Aut(G).$$

For any $G\in \NNG,$ let $H=(H_V,H_E,H_F)$ be a cellular automorphism of $G.$ Since for any vertex $v\in V,$ $\Phi(v)=\Phi(H_V(v)),$ 
$H_V: T_G\to T_G.$ This yields a group homomorphism,
\begin{equation}\label{eq:rho}\rho: \wt{\Aut}(G)\to S_{T_G}, \ H\mapsto H_V|_{T_G}\end{equation} where $S_{T_G}$ is the permutation group on $T_G$ and $H_V|_{T_G}$ is the restriction of $H_V$ to $T_G.$ The kernel of the homomorphism $\rho,$ denoted by $\ker \rho$, consists of $H\in \wt{\Aut}(G)$ such that $H_V|_{T_G}$ is the identity map on $T_G.$ By the group isomorphism theorem,
$$ \wt{\Aut}(G)/\ker\rho\cong \mathrm{im}(\rho)\leq S_{T_G},$$ where $\mathrm{im}(\rho)$ denotes the image of the map $\rho$ in $S_{T_G}.$ To show that  $\wt{\Aut}(G)$ is a finite group, it suffices to prove that $\ker\rho$ is finite.

We need some basic properties of cellular automorphisms of planar graphs. Recall that for $\sigma_1,\sigma_2\in F,$ we denote by $\sigma_1\sim \sigma_2$ if there is an edge $e$ such that $e\prec\sigma_1$ and $e\prec\sigma_2.$ By the definition of planar tessellation in the introduction, the edge $e$ satisfying the above property is unique. For any face $\sigma\in F,$ we write the \emph{vertex boundary} of $\sigma$ as $$\partial \sigma=\{v_1,v_2,\cdots, v_{\deg(\sigma)}\}$$ such that $v_i\sim v_{i+1}$ for any $1\leq i\leq \deg(\sigma)$ by setting $v_{\deg(\sigma)+1}=v_1.$ For a cellular automorphism $H$ of a planar graph $G,$ we say that $H$ \emph{fixes a face} $\sigma$ if $H_V(v)=v, H_E(e)=e$ and $H_F(\sigma)=\sigma$ for any $v\in \partial\sigma, e\prec\sigma, e\in E.$ It is easy to check that $H$ fixes $\sigma$ if and only if $H_V(v)=v$ for any $v\in\partial\sigma.$ The following lemma is useful.
\begin{lemma}\label{lem:isom} Let $G=(V,E,F)$ be a planar graph and $H\in \wt{\Aut}(G).$ Suppose that there are a face $\sigma$ and $\{v_1,v_2\}\subset \partial \sigma$ with $v_1\sim v_2$ satisfying
$$H_F(\sigma)=\sigma, H_V(v_1)=v_1, H_V(v_2)=v_2,$$ then $H$ fixes the face $\sigma,$ i.e. $H_V(v)=v$ for any $v\in \partial \sigma.$ Moreover, $H$ is the identity map on $G.$ 
\end{lemma}
\begin{proof} For the first assertion, by $H_F(\sigma)=\sigma$ we have $H_V(\partial \sigma)=\partial \sigma.$ We write $\partial \sigma=\{v_1,v_2,\cdots, v_{\deg(\sigma)}\}$ such that $v_i\sim v_{i+1}$ for any $1\leq i\leq \deg(\sigma)$ by setting $v_{\deg(\sigma)+1}=v_1.$ Noting that $H_V(v_1)=v_1$ and $H_V(v_2)=v_2,$ we get that $H_V(v_3)=v_3$ since the cellular automorphism $H$ preserves the incidence structure. The result follows from the induction argument on $v_i$ for $i\geq 3.$ 

For the second assertion, it suffices to show that $H$ fixes any face $\tilde{\sigma}\in F.$ Since the dual graph of the planar graph $G$ is connected, there is a sequence of faces $\{\sigma_j\}_{j=0}^M\subset F,$ $M\geq 1$, such that 
$$\sigma=\sigma_0\sim \sigma_1\sim \cdots\sim \sigma_M=\tilde{\sigma}.$$ By the above result, $H$ fixes the face $\sigma_0=\sigma.$ Let $\{v,w\}=\partial \sigma_0\cap  \partial \sigma_1.$ Noting that $H_F(\sigma_0)=\sigma_0, H_V(v)=v$ and $H_V(w)=w,$ we have $H_F(\sigma_1)=\sigma_1$ by the incidence preserving property of $H.$ Then one applies the first assertion to $\sigma_1,$ and yields that $H$ fixes $\sigma_1.$ Similarly, the induction argument implies that $H$ fixes $\sigma_j$ for any $0\leq j\leq M.$ Hence $H$ fixes $\tilde{\sigma}$ which proves the second assertion.

\end{proof} 

For any $v\in T_G,$ we denote by $N(v)$ the set of neighbours of $v$ and by $S_{N(v)}$ the permutation group on $N(v).$ For any $H\in \ker\rho,$ it is easy to see that $H_V:N(v)\to N(v)$ since $H_V(v)=v.$ This induces a group homomorphism 
\begin{equation}\label{eq:pi}\pi: \ker\rho\to S_{N(v)}, \ H\mapsto H_V|_{N(v)}.\end{equation}
\begin{lemma}\label{lem:isom2}Let $G=(V,E,F)$ be a planar graph with nonnegative combinatorial curvature and positive total curvature. For any $v\in T_G,$ the map $\pi,$ defined as in \eqref{eq:pi}, is a group monomorphism.
\end{lemma}
\begin{proof}It suffices to show that the map $\pi$ is injective. For any $H\in \ker\pi,$ $H_V(w)=w,$ for any $w\in N(v)\cup\{v\}.$ Let $w_1,w_2\in N(v)$ such that there is a face $\sigma$ such that $\{w_1,w_2,v\}\subset \partial \sigma.$ By the incidence preserving property of $H,$ $H_F(\sigma)=\sigma.$ Then Lemma~\ref{lem:isom} yields that $H$ is the identity on $G.$ This implies that $\ker\pi$ is trivial and $\pi$ is injective.
\end{proof}

Now we are ready to estimate the order of the automorphism group of a planar graph with nonnegative curvature and positive total curvature.
\begin{theorem}\label{thm:autoest}Let $G=(V,E,F)$ be a planar graph with nonnegative combinatorial curvature and positive total curvature. Then the automorphism group of $G$ is finite. Set $Q(G):=\sharp\wt{\Aut}(G),$ $a:=\sharp T_G$ and $b:=\max_{v\in T_G}\deg(v).$ We have the following: 
\begin{enumerate}
\item If $D_G\leq 42,$
$$Q(G)\ |\ a! b!.$$ 
\item If $D_G>42,$ then 
$$Q(G)\ \mathrm{divides}\left\{\begin{array}{ll}2D_G,& G\ \mathrm{is\ infinite},\\4D_G,& G\ \mathrm{is\ finite}. \end{array}\right.$$%$$Q(G)\ |\ 2D_G.$$
\end{enumerate}
\end{theorem}

\begin{proof} Suppose that $D_G\leq 42,$ then by Lemma~\ref{lem:isom},
$$\sharp({\wt{Aut}(G)}/{\ker \rho})\ |\ a!,$$ where $\rho$ is defined in \eqref{eq:rho}. By Lemma~\ref{lem:isom2}, $$\sharp\ker\rho\ |\ b!.$$ This yields the result.

Suppose that $D_G> 42,$ then the planar graph $G$ has some special structure. For simplicity, we write $m:=D_G.$ We divide it into two cases:

\begin{description}
\item[Case 1.] $G$ is infinite. In this case, $G$ is a prism-like graph. By Theorem~\ref{thm:prminf}, there is only one face $\sigma$ with $\deg(\sigma)=m>42,$ and $T_G=\partial \sigma.$ For any $H\in \ker\rho,$ $H_F(\sigma)=\sigma$ and $H_V(v)=v$ for any $v\in T_G,$ so that by Lemma~\ref{lem:isom}, $H$ is the identity map on $G.$ This yields that $$\wt{\Aut}(G)\cong \mathrm{im}(\rho)\leq S_{T_G}.$$ For any $H\in \wt{\Aut}(G),$ $\rho(H)\in \mathrm{im}(\rho)$ induces a graph isomorphism on the cycle graph $C_{m},$ which is given by the dihedral group $D_{m}.$ That yields that $$\mathrm{im}(\rho)\leq D_{m} \mathrm{\ and\ } Q(G)|2m.$$

\item[Case 2.] $G$ is finite. We have two subcases.

\begin{description}
\item[Subcase 2.1.] Suppose that $\sharp\{\sigma\in F: \deg(\sigma)>42\}\geq 2,$ then $G$ is a finite prism-like graph.
By Theorem~\ref{thm:finiteprism}, there are exactly two faces $\sigma_1$ and $\sigma_2,$ which are disjoint, of same degree $m>42,$ and $T_G=\partial\sigma_1\cup\partial \sigma_2.$ For any $H\in \wt{\Aut}(G),$ $H_F:\{\sigma_1,\sigma_2\}\to \{\sigma_1,\sigma_2\},$ which yields a group homomorphism
$$\wt{\rho}: \wt{\Aut}(G)\to S_{\{\sigma_1,\sigma_2\}},\ H\mapsto H_F|_{\{\sigma_1,\sigma_2\}},$$ where 
$S_{\{\sigma_1,\sigma_2\}}$ is the permutation group on ${\{\sigma_1,\sigma_2\}}.$  This implies that
\begin{equation}\label{eq:par1}\sharp ( \wt{\Aut}(G)/\ker\wt{\rho})=1\ \mathrm{or}\ 2.\end{equation} For any $H\in \ker\wt{\rho},$ $H_F(\sigma_i)=\sigma_i$ for $i=1,2.$ Hence $H_V:\partial \sigma_1\to \partial \sigma_1.$ This yields a group homomorphism
$$\eta: \ker\wt{\rho}\to S_{\partial \sigma_1}, H\mapsto H_V|_{\partial \sigma_1},$$ where $S_{\partial \sigma_1}$ is the permutation group on ${\partial \sigma_1}.$ The same argument as in Case 1 implies that $\ker\eta$ is trivial and $\mathrm{im}(\eta)\leq D_m.$ Hence $$\sharp\ker\wt{\rho}\ | 2m.$$ By combining it with \eqref{eq:par1} we have $$Q(G)| 4m.$$

\item[Subcase 2.2.] Suppose that $\sharp\{\sigma\in F: \deg(\sigma)>42\}=1,$ then there is only one face $\sigma\in F$ such that $\deg(\sigma)=m.$ The same argument as in Case 1 yields that
$$Q(G)| 2m.$$
\end{description}
\end{description}

Combining all cases above, we prove the theorem.

\end{proof}

By combining the estimates of the size of $T_G$ in Theorem~\ref{main thm} and Theorem~\ref{thm:finite} with the above result, we obtain the estimates for the orders of cellular automorphism groups.
\begin{proof}[Proof of Theorem~\ref{thm:autoint}]
For any $G\in\NNG$ with $D_G\leq 42,$ we obtain that $\sharp T_G\leq 132$ if $G$ is infinite by Theorem~\ref{main thm}, and $\sharp T_G\leq 264$ if $G$ is finite by Theorem~\ref{main thm}. Then the theorem follows from Theorem~\ref{thm:autoest}.
\end{proof}

%As a corollary of Theorem~\ref{??} and the above result, we get the upper bound estimate for the order of automorphism group of a planar graph with nonnegative combinatorial curvature.
%\begin{corollary}Let $G=(V,E,F)$ be a planar graph with nonnegative combinatorial curvature. Let $Q(G)$ denote the order of the automorphism group of $G$ and $D_G:=\max_{\sigma\in F}\deg(\sigma).$ Then we have
%\begin{enumerate}\item 
%If $D_G\leq 42,$ $Q(G)\leq 132!\times 5!.$
%\item If $D_G>42,$ $Q(G)\leq 2D_G.$
%\end{enumerate}
%\end{corollary}

%\begin{figure}[tb]
%\begin{minipage}[b]{\textwidth}
%\centering
%\includegraphics[width=7.18cm,height=4.08cm]{3d12-half.pdf}
%\end{minipage}
%\caption{\small A graph with boundary and with total curvature $\frac{3}{12}$.}
% \label{fig-half3}
%\end{figure}

%\begin{figure}[tb]
%\begin{minipage}[b]{\textwidth}
%\centering
%\includegraphics[width=4.68cm,height=3.964cm]{4d12-half.pdf}
%\end{minipage}
%\caption{\small A graph with boundary and with total curvature $\frac{4}{12}$.}
% \label{fig-half4}
%\end{figure}

%\begin{figure}[tb]
%\begin{minipage}[b]{\textwidth}
%\centering
%\includegraphics[width=4.5cm,height=0.47cm]{6d12-half.pdf}
%\end{minipage}
%\caption{\small A graph with boundary and with total curvature $\frac{6}{12}$.}
% \label{fig-half6}
%\end{figure}

%%%%%%%%%%%%%%%%%%%%%%%%%%%%%%%%%%%%%%%%%%%%%%%%%%%%%%%%%%%%%%%%%%%%%%%
%%%%%%%%%%%%%%%%%%%%%%%%%%%%%%%%%%%%%%%%%%%%%%%%%%%%%%%%%%%%%%%%%%%%%%%

\bigskip

{\bf Acknowledgements.} 
%We thank the anWe thank the anonymous referee for his valuable comments and suggestions
We thank the anonymous referee for carefully reading the proofs, and providing many valuable comments and suggestions to improve the writing of the paper. We thank Wenxue Du for many helpful discussions on automorphism groups of planar graphs. 

B. H. is supported by NSFC, grant no. 11831004 and grant no. 11826031. Y. S. is supported by NSFC, grant no. 11771083 and NSF of Fujian Province through Grants 2017J01556, 2016J01013.

\bibliography{Reti-ref-sec}

\newcommand{\etalchar}[1]{$^{#1}$}
\begin{thebibliography}{KHO{\etalchar{+}}85}

\bibitem[AH77]{MR0543792}
K.~Appel and W.~Haken.
\newblock Every planar map is four colorable. {I}. {D}ischarging.
\newblock {\em Illinois J. Math.}, 21(3):429--490, 1977.

\bibitem[Ale05]{MR2127379}
A.~D. Alexandrov.
\newblock {\em Convex polyhedra}.
\newblock Springer Monographs in Mathematics. Springer-Verlag, Berlin, 2005.
\newblock Translated from the 1950 Russian edition by N. S. Dairbekov, S. S.
  Kutateladze and A. B. Sossinsky, With comments and bibliography by V. A.
  Zalgaller and appendices by L. A. Shor and Yu. A. Volkov.

\bibitem[Bab75]{MR0371715}
L.~Babai.
\newblock Automorphism groups of planar graphs. {II}.
\newblock pages 29--84. Colloq. Math. Soc. J\'anos Bolyai, Vol. 10, 1975.

\bibitem[BBI01]{MR1835418}
D.~Burago, Yu. Burago, and S.~Ivanov.
\newblock {\em A course in metric geometry}, volume~33 of {\em Graduate Studies
  in Mathematics}.
\newblock American Mathematical Society, Providence, RI, 2001.

\bibitem[BD97]{MR1441972}
G.~Brinkmann and A.~W.~M. Dress.
\newblock A constructive enumeration of fullerenes.
\newblock {\em J. Algorithms}, 23(2):345--358, 1997.

\bibitem[BE17a]{BuE17b}
V.~Buchstaber and N.~Erokhovets.
\newblock {Finite sets of operations sufficient to construct any fullerene from
  $C_{20}$}.
\newblock {\em Structural Chemistry}, 28(1):225--234, 2017.

\bibitem[BE17b]{MR3706860}
V.~M. Buchstaber and N.~Yu. Erokhovets.
\newblock Constructions of families of three-dimensional polytopes,
  characteristic patches of fullerenes, and {P}ogorelov polytopes.
\newblock {\em Izv. Ross. Akad. Nauk Ser. Mat.}, 81(5):15--91, 2017.

\bibitem[BGM12]{BGM12}
G.~Brinkmann, J.~Goedgebeur, and B.~D. McKay.
\newblock {The generation of fullerenes}.
\newblock {\em J. Chem. Inf. Model.}, 52(11):2910--2918, 2012.

\bibitem[BP01]{BP01}
O.~Baues and N.~Peyerimhoff.
\newblock Curvature and geometry of tessellating plane graphs.
\newblock {\em Discrete Comput. Geom.}, 25(1):141--159, 2001.

\bibitem[BP06]{MR2243299}
O.~Baues and N.~Peyerimhoff.
\newblock Geodesics in non-positively curved plane tessellations.
\newblock {\em Adv. Geom.}, 6(2):243--263, 2006.

\bibitem[CBGS08]{MR2410150}
J.~H. Conway, H.~Burgiel, and C.~Goodman-Strauss.
\newblock {\em The symmetries of things}.
\newblock A K Peters, Ltd., Wellesley, MA, 2008.

\bibitem[CC08]{MR2410938}
B.~Chen and G.~Chen.
\newblock Gauss-{B}onnet formula, finiteness condition, and characterizations
  of graphs embedded in surfaces.
\newblock {\em Graphs Combin.}, 24(3):159--183, 2008.

\bibitem[Che09]{MR2470818}
B.~Chen.
\newblock The {G}auss-{B}onnet formula of polytopal manifolds and the
  characterization of embedded graphs with nonnegative curvature.
\newblock {\em Proc. Amer. Math. Soc.}, 137(5):1601--1611, 2009.

\bibitem[DM07]{MR2299456}
M.~DeVos and B.~Mohar.
\newblock An analogue of the {D}escartes-{E}uler formula for infinite graphs
  and {H}iguchi's conjecture.
\newblock {\em Trans. Amer. Math. Soc.}, 359(7):3287--3300, 2007.

\bibitem[Gal09]{Gal09}
B.~Galebach.
\newblock {$n$-Uniform tilings}.
\newblock {\em Available online at http://probabilitysports.com/tilings.html},
  2009.

\bibitem[Ghi17]{Gh17}
L.~Ghidelli.
\newblock On the largest planar graphs with everywhere positive combinatorial
  curvature.
\newblock {\em arXiv:1708.08502}, 2017.

\bibitem[Gro87]{MR919829}
M.~Gromov.
\newblock Hyperbolic groups.
\newblock In {\em Essays in group theory}, volume~8 of {\em Math. Sci. Res.
  Inst. Publ.}, pages 75--263. Springer, New York, 1987.

\bibitem[GS89]{MR992195}
B.~Gr\"unbaum and G.~C. Shephard.
\newblock {\em Tilings and patterns}.
\newblock A Series of Books in the Mathematical Sciences. W. H. Freeman and
  Company, New York, 1989.

\bibitem[Hig01]{MR1864922}
Y.~Higuchi.
\newblock Combinatorial curvature for planar graphs.
\newblock {\em J. Graph Theory}, 38(4):220--229, 2001.

\bibitem[HJL02]{MR1894115}
O.~H\"aggstr\"om, J.~Jonasson, and R.~Lyons.
\newblock Explicit isoperimetric constants and phase transitions in the
  random-cluster model.
\newblock {\em Ann. Probab.}, 30(1):443--473, 2002.

\bibitem[HJL15]{MR3318509}
B.~Hua, J.~Jost, and S.~Liu.
\newblock Geometric analysis aspects of infinite semiplanar graphs with
  nonnegative curvature.
\newblock {\em J. Reine Angew. Math.}, 700:1--36, 2015.

\bibitem[HL16]{MR3547929}
B.~Hua and Y.~Lin.
\newblock Curvature notions on graphs.
\newblock {\em Front. Math. China}, 11(5):1275--1290, 2016.

\bibitem[HS03]{MR2038013}
Y.~Higuchi and T.~Shirai.
\newblock Isoperimetric constants of {$(d,f)$}-regular planar graphs.
\newblock {\em Interdiscip. Inform. Sci.}, 9(2):221--228, 2003.

\bibitem[HS17]{HuaSu17a}
B.~Hua and Y.~Su.
\newblock {Total curvature of planar graphs with nonnegative curvature}.
\newblock {\em arXiv:1703.04119}, 2017.

\bibitem[Imr75]{MR0384588}
W.~Imrich.
\newblock On {W}hitney's theorem on the unique embeddability of {$3$}-connected
  planar graphs.
\newblock pages 303--306. (loose errata), 1975.

\bibitem[Ish90]{Ishida90}
M.~Ishida.
\newblock Pseudo-curvature of a graph.
\newblock In {\em lecture at ��Workshop on topological graph theory��}.
  Yokohama National University, 1990.

\bibitem[Kel10]{MR2558886}
M.~Keller.
\newblock The essential spectrum of the {L}aplacian on rapidly branching
  tessellations.
\newblock {\em Math. Ann.}, 346(1):51--66, 2010.

\bibitem[Kel11]{MR2826967}
M.~Keller.
\newblock Curvature, geometry and spectral properties of planar graphs.
\newblock {\em Discrete Comput. Geom.}, 46(3):500--525, 2011.

\bibitem[KHO{\etalchar{+}}85]{KHBCS}
H.~W. Kroto, J.~R. Heath, S.~C. O'Brien, R.~F. Curl, and R.~E. Smalley.
\newblock {$C_{60}$ : Buckminsterfullerene}.
\newblock {\em Nature}, 318:162--163, 1985.

\bibitem[KP11]{MR2818734}
M.~Keller and N.~Peyerimhoff.
\newblock Cheeger constants, growth and spectrum of locally tessellating planar
  graphs.
\newblock {\em Math. Z.}, 268(3-4):871--886, 2011.

\bibitem[LPZ02]{MR1923955}
Serge Lawrencenko, Michael~D. Plummer, and Xiaoya Zha.
\newblock Isoperimetric constants of infinite plane graphs.
\newblock {\em Discrete Comput. Geom.}, 28(3):313--330, 2002.

\bibitem[Man71]{MR0296808}
P.~Mani.
\newblock Automorphismen von polyedrischen {G}raphen.
\newblock {\em Math. Ann.}, 192:279--303, 1971.

\bibitem[Moh88]{MR923264}
B.~Mohar.
\newblock Embeddings of infinite graphs.
\newblock {\em J. Combin. Theory Ser. B}, 44(1):29--43, 1988.

\bibitem[Moh02]{MR1897390}
B.~Mohar.
\newblock Light structures in infinite planar graphs without the strong
  isoperimetric property.
\newblock {\em Trans. Amer. Math. Soc.}, 354(8):3059--3074, 2002.

\bibitem[Nev70]{MR0279280}
R.~Nevanlinna.
\newblock {\em Analytic functions}.
\newblock Translated from the second German edition by Phillip Emig. Die
  Grundlehren der mathematischen Wissenschaften, Band 162. Springer-Verlag, New
  York-Berlin, 1970.

\bibitem[NS11]{MR2836763}
R.~Nicholson and J.~Sneddon.
\newblock New graphs with thinly spread positive combinatorial curvature.
\newblock {\em New Zealand J. Math.}, 41:39--43, 2011.

\bibitem[Oh17]{MR3624614}
B.-G. Oh.
\newblock On the number of vertices of positively curved planar graphs.
\newblock {\em Discrete Math.}, 340(6):1300--1310, 2017.

\bibitem[Old17]{Ol17}
P.~R. Oldridge.
\newblock Characterizing the polyhedral graphs with positive combinatorial
  curvature.
\newblock {\em thesis, available at
  https://dspace.library.uvic.ca/handle/1828/8030}, 2017.

\bibitem[RBK05]{RBK05}
T.~R{\'e}ti, E.~Bitay, and Z.~Kosztol{\'a}nyi.
\newblock {On the polyhedral graphs with positive combinatorial curvature}.
\newblock {\em Acta Polytechnica Hungarica}, 2(2):19--37, 2005.

\bibitem[RSST97]{MR1441258}
N.~Robertson, D.~Sanders, P.~Seymour, and R.~Thomas.
\newblock The four-colour theorem.
\newblock {\em J. Combin. Theory Ser. B}, 70(1):2--44, 1997.

\bibitem[SS98]{MR1739919}
B.~Servatius and H.~Servatius.
\newblock Symmetry, automorphisms, and self-duality of infinite planar graphs
  and tilings.
\newblock In {\em International {S}cientific {C}onference on {M}athematics.
  {P}roceedings (\v Zilina, 1998)}, pages 83--116. Univ. \v Zilina, \v Zilina,
  1998.

\bibitem[Sto76]{MR0410602}
D.~A. Stone.
\newblock A combinatorial analogue of a theorem of {M}yers.
\newblock {\em Illinois J. Math.}, 20(1):12--21, 1976.

\bibitem[SY04]{MR2096789}
L.~Sun and X.~Yu.
\newblock Positively curved cubic plane graphs are finite.
\newblock {\em J. Graph Theory}, 47(4):241--274, 2004.

\bibitem[Tho82]{MR685062}
C.~Thomassen.
\newblock Duality of infinite graphs.
\newblock {\em J. Combin. Theory Ser. B}, 33(2):137--160, 1982.

\bibitem[Thu98]{MR1668340}
W.~P. Thurston.
\newblock Shapes of polyhedra and triangulations of the sphere.
\newblock In {\em The {E}pstein birthday schrift}, volume~1 of {\em Geom.
  Topol. Monogr.}, pages 511--549. Geom. Topol. Publ., Coventry, 1998.

\bibitem[Whi33]{MR1506961}
H.~Whitney.
\newblock 2-{I}somorphic {G}raphs.
\newblock {\em Amer. J. Math.}, 55(1-4):245--254, 1933.

\bibitem[Woe98]{Woess98}
W.~Woess.
\newblock A note on tilings and strong isoperimetric inequality.
\newblock {\em Math. Proc. Camb. Phil. Soc.}, 124:385--393, 1998.

\bibitem[\.Z97]{MR1600371}
A.~\.Zuk.
\newblock On the norms of the random walks on planar graphs.
\newblock {\em Ann. Inst. Fourier (Grenoble)}, 47(5):1463--1490, 1997.

\bibitem[Zha08]{MR2466966}
L.~Zhang.
\newblock A result on combinatorial curvature for embedded graphs on a surface.
\newblock {\em Discrete Math.}, 308(24):6588--6595, 2008.

\end{thebibliography}
\bibliographystyle{alpha}

\end{document}